\documentclass{amsart}

\usepackage{esdiff}
\usepackage{color}

\newcommand{\R}{\mathbb{R}}

\newcommand{\s}{\mathbb{S}}

\newcommand{\conbod}{\mathcal{K}^n}
\newcommand{\Mel}[2]{\mathcal{M}_{#1}(#2)}

\newcommand{\vol}{\text{\rm Vol}}

\newtheorem{conj}{Conjecture}[section]
\newtheorem{theorem}[conj]{Theorem}

\newtheorem{lemma}[conj]{Lemma}
\newtheorem{proposition}[conj]{Proposition}
\newtheorem{corollary}[conj]{Corollary}

\newtheorem{definition}[conj]{Definition}

\DeclareMathOperator{\supp}{supp}

\let\oldtocsection=\tocsection

\let\oldtocsubsection=\tocsubsection

\let\oldtocsubsubsection=\tocsubsubsection

\renewcommand{\tocsection}[2]{\hspace{0em}\oldtocsection{#1}{#2}}
\renewcommand{\tocsubsection}[2]{\hspace{1em}\oldtocsubsection{#1}{#2}}
\renewcommand{\tocsubsubsection}[2]{\hspace{2em}\oldtocsubsubsection{#1}{#2}}

\usepackage[style=numeric-comp,url=false,doi=false, bibencoding=utf8,giveninits=true,maxnames=50]{biblatex}
\DeclareNameAlias{default}{family-given}
\bibliography{references}

\begin{document}
\title[Weighted Berwald's Inequality]{Weighted Berwald's Inequality}                                 
    \author[Langharst]{Dylan Langharst}                                
    \address{Dylan Langharst \\ Institut de Math\'ematiques de Jussieu, Sorbonne Universit\'e\\
Paris, 75252 France}                                                                
    \email{dylan.langharst@imj-prg.fr}   

      \author[Putterman]{Eli Putterman}                                
    \address{Eli Putterman \\ School of Mathematical Sciences, 
Tel Aviv University \\ Tel Aviv 66978, Israel}                                   \email{putterman@mail.tau.ac.il} 
    
    \thanks{The first named author was supported in part by the U.S. National Science Foundation Grant DMS-2000304 and the United States - Israel Binational Science Foundation (BSF) Grant 2018115. This work was completed while the first named author was a postdoctoral researcher funded by a Fondation Sciences Math\'ematiques de Paris fellowship. Both authors were supported during Fall 2022 by the National Science Foundation under Grant DMS-1929284 while in residence at the Institute for Computational and Experimental Research in Mathematics in Providence, RI, during ``Harmonic Analysis and Convexity" program and during Spring 2024 by the Hausdorff Research Institute for Mathematics in Bonn, Germany, while in residence for the Dual Trimester Program: "Synergies between modern probability, geometric analysis and stochastic geometry".}
    \keywords{Berwald's Inequality, Projection Bodies, Radial Mean Bodies, Zhang's Inequality, Petty Projection Inequality.}                                  
    \subjclass{52A39, 52A41; 28A75}                                  
    \begin{abstract}
The inequality of Berwald is a reverse-H{\"o}lder like inequality for the $p$th average, $p\in (-1,\infty),$ of a non-negative, concave function over a convex body in $\R^n.$ We prove Berwald's inequality for averages of functions with respect to measures that have some concavity conditions, e.g. $s$-concave measures, $s\in \mathbb{R}.$ We also obtain equality conditions; in particular, this provides a new proof for the equality conditions of the classical inequality of Berwald. As applications, we generalize a number of classical bounds for the measure of the intersection of a convex body with a half-space and also the concept of radial means bodies and the projection body of a convex body.
\end{abstract}          
    \maketitle  
    \tableofcontents

\newpage
\section{Introduction}
Let $\R^n$ be the standard $n$-dimensional real vector space with the Euclidean structure. We write $\vol_m(C)$ for the $m$-dimensional Lebesgue measure (volume) of a measurable set $C \subset \R^n$, where $m = 1, . . . , n$ is the dimension
of the minimal affine space containing $C$. The volume of the unit ball $B_2^n$ is written as $\kappa_n,$ and its boundary, the unit sphere, will be denoted as usual $\s^{n-1}.$ A set $K\subset \R^n$ is said to be \textit{convex} if for every $x,y\in K$ and $\lambda\in[0,1],$ $(1-\lambda)x+\lambda y\in K.$ We say $K$ is a convex body if it is a convex, compact set with non-empty interior; the set of all convex bodies in $\R^n$ will be denoted by $\conbod$. The set of those convex bodies containing the origin will be denoted $\conbod_0.$ A convex body $K$ is centrally symmetric, or just symmetric, if $K=-K$. There exists an addition on the set of convex bodies: the Minkowski sum of $K$ and $L$, and one has that $K+L=\{ a + b: a\in K, b\in L\}.$ 

We recall that a non-negative function $f$ is said to be \textit{concave} on $\R^n$ if for every $x,y\in\R^n$ and $\lambda\in [0,1]$ one has
$$f((1-\lambda)x+\lambda y) \geq (1-\lambda)f(x) + \lambda f(y),$$ and that the \textit{support} of a function is precisely $\text{supp}(f)=\overline{\{x\in\R^n:f(x) >0\}}.$ One can see that a non-negative, concave function will be supported on a convex set. It is easy to show if a non-negative, concave function takes the value infinity anywhere on its support, then the function is identically infinity on the interior of its support from convexity; therefore, throughout this paper, given a non-negative, concave function $f,$ we shall assume it is not identically infinity, and so $f$ will have a finite maximum value, denoted $\|f\|_\infty$. If $K$ is the support of a non-negative, concave function $f$, then $K_t=\{x\in\R^n:f(x) \geq t\}=\{f\geq t\}$ are the \textit{level sets} of $f$. Notice that the level sets are also convex. Additionally, if $\|f\|_\infty=f(0),$ then $0\in K_t$ for all $t \leq \|f\|_\infty.$ If $f$ is even, then $K$ is symmetric and so too is each $K_t.$ In any case, if $K$ is also bounded, then each $K,K_t\in\conbod$ (for each $t\leq \|f\|_\infty$).

We next recall that the classical Berwald inequality states that if $f$ is a non-negative, concave function supported on some convex set $K\subset \R^n,$ then, the function given by

\begin{equation}\label{classical_Berwald} t_f(p)=\left({\binom{n+p}{p}}\frac{1}{\vol_n(K)}\int_Kf^p(x)dx\right)^{1/p} \end{equation}
is decreasing for $p\in (-1,\infty)$ \cite{Berlem} with equality \cite{GZ98} if and only if the graph of $f$ is a certain cone with $K$ as its base. Here, the combinatorial coefficients are given by $\binom{m}{p}=\frac{\Gamma(m+1)}{\Gamma(p+1)\Gamma(m-p+1)},$ with $\Gamma(z)$ the standard Gamma function, defined for $z \in \mathbb{C}$ except for when $z$ is negative integer. Usually written in the form $t_f(q)\leq t_f(p)$ for $-1 < p \leq q < \infty,$ Berwald's inequality has several applications in the fields of convex geometry and probability theory, see for example \cite{GZ98,MP89,BM12,GP99}. The first goal of this paper is to establish generalizations of Berwald's inequality for measures with density under certain concavity assumptions.  We will also analyze equality conditions; in particular, we obtain equality conditions for the classical Berwald inequality by a method independent of other proofs (e.g., \cite{AAGJV,GZ98,Bor73}). To accomplish these tasks, we first
prove a generalized Berwald's inequality, Lemma~\ref{l:mel_ber}.

We will say a Borel measure $\mu$ has \textit{density} if it has a locally integrable Radon-Nikodym derivative from $\R^n$ to $\R^{+}$, i.e,
\[
\frac{d\mu(x)}{dx} = \phi(x), \text{ with } \phi \colon \R^n \to \R^+,\phi\in L^1_{\text{loc}}(\R^n). 
\]
A Borel measure $\mu$ on $\R^n$ is said to be $F$-concave on a class $\mathcal{C}$ of compact subsets of $\R^n$ if there exists a continuous, (strictly) monotonic, invertible function $F:(0,\mu(\R^n))\to (-\infty,\infty)$ such that, for every pair $A,B \in \mathcal{C}$  and every $t \in [0,1]$, one has
  $$\mu(t A +(1-t)B) \geq F^{-1}\left(tF(\mu(A)) +(1-t)F(\mu(B))\right).$$
 When $F(x)=x^s, s > 0$ this can be written as
	$$\mu(t A +(1-t)B)^{s}\geq t\mu(A)^s +(1-t)\mu(B)^{s},$$
 and we say $\mu$ is $s$-concave. When $s=1$, we merely say the measure is concave.	In the limit as $s\rightarrow 0$, we obtain the case of log-concavity, which can also be obtained by taking $F(x) = \log x$:
	$$\mu(t A +(1-t)B)\geq\mu(A)^{t}\mu(B)^{1-t}.$$
	 The classical Brunn-Minkowski inequality  (see for example \cite{gardner_book}) asserts the  $1/n$-concavity of the Lebesgue measure on the class of all compact subsets of $\R^n$. From Borell's classification on concave measures \cite{Bor75}, a Radon measure (locally finite and regular Borel measure) is log-concave on Borel subsets of $\R^n$ if, and only if,  $\mu$ has a density $\phi(x)$ that is log-concave, i.e. $\phi(x)=Ae^{-\psi(x)},$ where $A>0$ and $\psi:\R^n\to\R^+$ is convex. Similarly, a Radon measure is $s$-concave on Borel subsets of $\R^n$, $s\in (-\infty,0)\cup(0,1/n),$ if, and only if, $\mu$ has a density $\phi(x)$ that is $p$-concave (if $s>0$) or $p$-convex (if $s<0$), where $p=s/(1-ns).$ However, all we will require is that a measure is $s$-concave on a class of convex sets; we will discuss an important example below. Thus, our results in the case of $s$-concave measures include measures beyond Borell's classification.
  
We can now state our first main result, which is the Berwald inequality for $F$-concave measures under different restrictions on the function $F$. This result applies to a variety of measures, including $s$-concave ones.

\begin{theorem}[The Berwald Inequality for measures with concavity]
\label{t:ber}
Let $f$ be a non-negative, concave function supported on $K\subset \R^n$. Let $\mu$ be a Borel measure such that $0< \mu(K) < \infty$ and $\mu$ satisfies one of the below listed concavity assumptions on a collection of convex subsets of $K$ containing the level sets of $f$. Then, for any $-1<p\leq q < p_{\max}$ we have
$$C(p,\mu,K)\left(\frac{1}{\mu(K)}\int_K f(x)^p d\mu(x)\right)^{1/p}\geq C(q,\mu,K)\left(\frac{1}{\mu(K)}\int_K f(x)^q d\mu(x)\right)^{1/q},$$
where
\begin{enumerate}
    \item If $\mu$ is $F$-concave, where $F:[0,\mu(K)]\to [0,\infty)$ is a continuous, increasing and invertible function: $C(p,\mu,K)
    =$\begin{align*}\begin{cases}
    \left(\frac{p}{\mu(K)}\int_{0}^{1} F^{-1}\left[F(\mu(K))(1-t)\right]t^{p-1}dt\right)^{-\frac{1}{p}} & \text {for } p>0 \\
    \left(\frac{p}{\mu(K)}\int_{0}^{1}t^{p-1} (F^{-1}\left[F(\mu(K))(1-t)\right]-\mu(K))dt +1\right)^{-\frac{1}{p}} & \text {for } p\in (-1,0).
    \end{cases}
    \end{align*}
    There is equality if, and only if, $F(0)=0,$ for all $t\in [0,\|f\|_\infty]$ the following formula holds
$$\mu(\{f\geq t\})=F^{-1}\left[F(\mu(K))\left(1-\frac{t}{\|f\|_{\infty}}\right)\right],$$ and for all $p\in (-1,\infty),$ $\|f\|_\infty$ must satisfy
$$\|f\|_\infty=C(p,\mu,K)\left(\frac{1}{\mu(K)}\int_K f(x)^p d\mu(x)\right)^{1/p}.$$

    \item If $\mu$ is $Q$-concave, where $Q:(0,\mu(K)]\to (-\infty,\infty)$ is a continuous, increasing and invertible function: $C(p,\mu,K)=$
    $$
    \begin{cases}
    \left(\frac{p}{\mu(K)}\int_0^\infty Q^{-1}\left[Q(\mu(K))-t\right]t^{p-1} dt\right)^{-\frac{1}{p}} & \text {for } p>0 \\
    \left(\frac{p}{\mu(K)}\int_{0}^{\infty}t^{p-1} (Q^{-1}\left[Q(\mu(K)-t)\right]-\mu(K))dt\right)^{-\frac{1}{p}} & \text {for } p\in (-1,0).
    \end{cases}
    $$
    Equality is never obtained.
    
    \item If $\mu$ is $R$-concave, where $R:(0,\mu(K)]\to (0,\infty)$ is a continuous, decreasing and invertible function: $C(p,\mu,K)=$
    $$
    \begin{cases}
    \left(\frac{p}{\mu(K)}\int_0^\infty R^{-1}\left[R(\mu(K))(1+t)\right]t^{p-1} dt\right)^{-\frac{1}{p}} & \text {for } p>0 \\
    \left(\frac{p}{\mu(K)}\int_{0}^{\infty}t^{p-1} (R^{-1}\left[R(\mu(K))(1+t)\right]-\mu(K))dt\right)^{-\frac{1}{p}} & \text {for } p\in (-1,0).
    \end{cases}
    $$
    Equality is never obtained.
\end{enumerate}
In all cases, $p_{\max}$ is defined implicitly via $p_{\max}=\sup\{p>0: T_f(p)<\infty\},$
where
$$T_f(p)=C(p,\mu,K)\left(\frac{1}{\mu(K)}\int_K f(x)^p d\mu(x)\right)^{1/p}.$$ $T_f(0)$ is defined via continuity.
\end{theorem}
  \noindent We remark that cases 2 and 3 of Theorem~\ref{t:ber} have a strict inequality due to the fact, for Case 2, that $Q^{-1}[Q(\mu(K))-t]$ being integrable implies $Q^{-1}(-\infty)=0,$ or $Q(0)=-\infty.$ On the other hand, we will show that if there is equality, then $|Q(0)|$ would be finite. Similar logic holds for Case 3. However, the inequality is asymptotically sharp as $f$ is made arbitrarily large on its support. 
  
  We obtain the following corollary for $s$-concave measures; the case where $s<0$ was previously done by Fradelizi, Gu{\'e}don and Pajor \cite{FGP14}, by modifying Borell's proof \cite{Bor73_2} of the classical inequality of Berwald. Presented in \cite{FLM20} is a proof for all $s\in \R,$ based on techniques from a work by Koldobsky, Pajor and Yaskin \cite{KPY08}. Both extensions do not mention equality conditions.

\begin{corollary}[The Berwald Inequality for $s$-concave measures]
\label{cor:ber_s}
Let $f$ be a non-negative concave function supported on $K\subset\R^n.$ Let $\mu$ be a Borel measure finite on $K$ that is $s$-concave, $s\in \R$, on a collection of convex subsets of $K$ containing the level sets of $f$. Then, for any $-1<p\leq q<\infty$ we have
$$\left(\frac{C(p,s)}{\mu(K)}\int_K f(x)^p d\mu(x)\right)^{1/p}\geq \left(\frac{C(q,s)}{\mu(K)}\int_K f(x)^q d\mu(x)\right)^{1/q},$$
where $$C(p,s)= \begin{cases}{{\frac{1}{s}+p} \choose p} & \text{for } s>0,
\\
\Gamma(p+1)^{-1} & \text{if } s=0,
\\
s\left(p+\frac{1}{s}\right){{-\frac{1}{s}}\choose p} & \text{for } s<0.
\end{cases}$$

\noindent For $s<0,$ we must restrict to $p\in (-1,-1/s)$ for integrability. 

 If $s>0,$ there is equality if, and only if, for all $t\in [0,\|f\|_\infty]$ and $p\in (-1,\infty):$
$$\mu(\{x\in K: f(x) \geq t\})=\mu(K)\left(1-\frac{t}{\|f\|_\infty}\right)^{1/s}$$ implying $$\|f\|^p_\infty={{\frac{1}{s}+p} \choose p}\frac{1}{\mu(K)}\int_K f(x)^p d\mu(x).$$
If $s=0$ or $s <0,$ equality is never obtained.
\end{corollary}
  
The equality conditions to Corollary~\ref{cor:ber_s} may seem a bit strange; we are able to obtain an exact formula for the function $f$ when the measure $\mu$ is $s$-concave and $1/s$-homogeneous, $s\in (0,1/n]$. Recall that a Borel measure $\mu$ is said to be $\alpha$-homogeneous for some $\alpha >0$ if $\mu(tA)=t^{\alpha}\mu(A)$ for all compact sets $A \subset \supp \mu$ and $t>0$ so that $tA \subset \supp \mu.$ If $\mu$ has density $\phi$, then one can check using the Lebesgue differentiation theorem that this implies that $\phi$ is $(\alpha-n)-$homogeneous. 

We say a set $L$ with $0 \in \text{int}(L)$ is star-shaped if every line passing through the origin crosses the boundary of $L$ exactly twice. We say $L$ is a star body 
if it is a compact, star-shaped set whose radial function $\rho_L:\R^n\setminus \{0\} \to \R,$ given by $\rho_L(y)=\sup\{\lambda:\lambda y\in L\},$ is continuous. For $K\in\conbod_0,$ the \textit{Minkowski functional} of $K$ is defined to be $\|y\|_K=\rho^{-1}_K(y)=\inf\{r>0:y\in rK\}.$ The Minkowski functional $\|\cdot\|_K$ of $K\in\conbod_0$ is a norm on $\R^n$ if $K$ is symmetric. If $x\in\R^n$ and $L \subset \mathbb R^n$ satisfy that $L-x$ is a star body, then the generalized radial function of $L$ at $x$ is defined by $\rho_L(x,y):=\rho_{L-x}(-y)$. Note that for every $K\in\conbod,$ $K-x$ is a star body for every $x\in\text{int}(K)$.

One gets the following formula for $\mu(K)$ when $\mu$ is $\alpha$-homogeneous, $\alpha>0$, and $K$ is a star body in $\R^n$. 
\begin{equation}
\begin{split}
    \mu(K)&=\int_{\s^{n-1}}\int_0^{\rho_K(\theta)}\phi(r\theta)r^{n-1}drd\theta
    \\
    &=\int_{\s^{n-1}}\phi(\theta)\int_{0}^{\rho_K(\theta)}r^{\alpha-1}drd\theta=\frac{1}{\alpha}\int_{\s^{n-1}}\phi(\theta)\rho_K^\alpha(\theta)d\theta.
    \end{split}
    \label{eq:star_form}
\end{equation}
Crucial to the statement of equality conditions, and our investigations henceforth, will be the \textit{roof function} associated to a star body $K$, which we define as $\ell_K(0)=1,\ell_K(x)=0$ for $x\neq K$ and, for $x\in K\setminus\{0\},$ $\ell_K(x)=\left(1-\frac{1}{\rho_K(x)}\right).$ In polar coordinates, $\ell_K(r\theta)$ becomes an affine function in $r$ for $r\in [0,\rho_K(\theta)]$:
\begin{equation}
\label{eq:lk}
\ell_K(r\theta)=\left(1-\frac{r}{\rho_K(\theta)}\right).
\end{equation}
Note that if $K\in\conbod_0,$ then we can also write
$\ell_K(x)=1-\|x\|_K$ for $x\in K$ and $0$ otherwise.  Observe that, for a non-negative, concave function supported on some $K\in\conbod_0$ one obtains for $\theta\in\s^{n-1}$ and $r\in [0,\rho_{K}(\theta)]$ that
\begin{equation}
\begin{split}f(r\theta)&=f\left(\left(\frac{r}{\rho_{K}(\theta)}\rho_{K}(\theta)+0\left(1-\frac{r}{\rho_{K}(\theta)}\right)\right)\theta\right)
\\
&\geq \frac{r}{\rho_{K}(\theta)}f(\rho_{K}(\theta)\theta) + f(0)\ell_{K}(r\theta) \geq f(0)\ell_{K}(r\theta);
\end{split}
\label{eq:supporting_line_bound}
\end{equation}
we will make liberal use of this bound throughout this work. Functions of the form $f(x)=M\ell_{K-x_0}(x-x_0)$ for some $M>0$ and $x_0\in K$ will also be referred to roof functions, with height $M$ and vertex $x_0$. The reason for this vocabulary will become clearer below.

Using \eqref{eq:star_form}, one can verify by hand that the function $\ell_K(x)$ satisfies, for $\mu$ an $s$-concave, $1/s$-homogeneous measure, that
$$\int_K\ell_K(x)^pd\mu(x)={{\frac{1}{s}+p}\choose \frac{1}{s}}^{-1}\mu(K).$$
Therefore, $\ell_K(x)$ yields equality in the Berwald inequality for $s$-concave measures, Corollary~\ref{cor:ber_s}, under the additional assumption that $\mu$ is $1/s$-homogeneous. The next theorem shows this is the only such function.

\begin{theorem}(The Berwald Inequality for $s$-concave, $1/s$-homogeneous measures)
\label{t:sberh}
Let $f$ be a non-negative, concave function supported on $K\subset \R^n$. Let $\mu$ be a Radon measure containing $K$ in its support that is $s$-concave, $1/s$-homogeneous for some $s\in (0,1/n]$. Then, for any $-1<p\leq q<\infty$ we have
$$\left({{\frac{1}{s}+p} \choose p}\frac{1}{\mu(K)}\int_K f(x)^p d\mu(x)\right)^{1/p}\geq \left({{\frac{1}{s}+q} \choose q}\frac{1}{\mu(K)}\int_K f(x)^q d\mu(x)\right)^{1/q}.$$
Suppose $\|f\|_\infty=f(0)$. Then,  there is equality if, and only if, $f(r\theta)$ is an affine function in $r.$ i.e. one has $f(x)=\|f\|_{\infty}\ell_{K}(x).$
\end{theorem}
\noindent In our applications below, we will always be considering functions whose maximum is obtained at the origin, and so the minor constraint on the equality conditions does not hinder us. We now prove the classical Berwald inequality with equality conditions. Favard first conjectured the inequality in one dimension, and Berwald verified the inequality for all dimensions \cite{Berlem}, without equality conditions. In fact, when $n=1,$ Berwald was able to show the inequality is true for $-1<p\leq q <\infty,$ and this was extended to all dimensions by Borell \cite{Bor73}. However, the generality of his technique makes it difficult to establish where equality occurs.

Gardner and Zhang \cite{GZ98} gave a different proof, which yields that equality is satisfied in the classical Berwald inequality precisely when the graph of $f$ is a certain cone with $K$ as a base, i.e. that $f$ is a roof function. In Corollary~\ref{cor:classic}, we obtain a proof using Theorem~\ref{t:sberh}, verifying that our techniques reduce to the known result. We must also mention that this result was also obtained in \cite[Theorem 7.2]{AAGJV} via a different technique. In that work, the roof function was defined via its graph in $\R^{n+1}.$ Specifically they constructed the roof function in the following way: given a convex set $K\subset \R^n$ (which will become the base of a hypercone), let $M>0$ be the height of the hypercone, and let $x_0\in K$ be the location of the projection of vertex of the hypercone. Then, the roof function with height $M$ and vertex $x_0$ is equivalently defined as the non-negative, concave function $f$ whose graph is given by
$$\{(x, t) \in K \times \R: 0 \leq t \leq f(x)\}=\text{conv}\left(K \times\{0\},\left\{\left(x_0, M\right)\right\}\right),$$
where $\text{conv}$ denotes the convex hull. From this formulation, we obtain an interesting formula for the level sets of a roof function $f:$ for $0\leq t \leq M,$ one has that $K_t=\frac{t}{M}x_0 +(1-\frac{t}{M})K.$

\begin{corollary}[The Classical Berwald Inequality]
\label{cor:classic}
Let $f$ be a non-negative, concave function supported on $K\in\conbod$. Then, for any $-1<p\leq q < \infty$ we have
$$\left({{n+p} \choose p}\frac{1}{\vol_n(K)}\int_K f(x)^p dx\right)^{1/p}\geq \left({{n+q} \choose q}\frac{1}{\vol_n(K)}\int_K f(x)^q dx\right)^{1/q}.$$
There is equality if, and only if, $f(r\theta)$ is an affine function in $r$ up to translation i.e. if $x_0$ is the point in $K$ where the maximum of $f$ is obtained, one has $f(x)=\|f\|_{\infty}\ell_{K-x_0}(x-x_0).$
\end{corollary}
\begin{proof}
The inequality follows immediately from Theorem~\ref{t:sberh}, as do the equality conditions if the maximum of $f$ is obtained at the origin. If the maximum of $f$ is not obtained at the origin, let $x_0$ be the point in $K$ where $f$ obtains its maximum. Let $g(x)=f(x+x_0)$ and $\widetilde{K}=K-x_0.$ Then, $g(x)$ is a concave function supported on $\widetilde{K}$ with maximum at the origin, and, for every $p\in (-1,0)\cup (0,\infty)$ 
$$\frac{1}{\vol_n(K)}\int_K f(x)^p dx=\frac{1}{\vol_n(\widetilde{K})}\int_{\widetilde{K}} g(x)^p dx.$$ 

\noindent Therefore, since there is equality in the inequality for the function $f$ and the convex body $K$ by hypothesis, there is equality in the inequality for the function $g$ and the convex body $\widetilde{K}$. Consequently, we have
$$g(x)=\|g\|_{\infty}\ell_{\widetilde{K}}(x).$$
Using that $f(x)=g(x-x_0)$ and $\|g\|_{\infty}=\|f\|_{\infty}$ yields the result.
\end{proof}

We next list two applications for the standard Gaussian measure on $\R^n,$ which we recall is given by $d\gamma_n(x)=\frac{1}{(2\pi)^{n/2}}e^{-|x|^2/2}dx.$ From Borell's classification, we see that the Gaussian measure is log-concave on $\R^n$ over any collection of compact sets closed under Minkowski summation. Thus, we can apply the second case of Corollary~\ref{cor:ber_s} and obtain a Berwald-type inequality for the Gaussian measure in this case. However, the Ehrhard inequality shows one can improve on the log-concavity of the Gaussian measure: 
For $0<t<1$ and Borel sets $K$ and $L$ in $\R^{n}$, we have
\begin{equation}\label{e:Ehrhard_ineq}
\Phi^{-1}\left(\gamma_{n}((1-t) K+tL)\right) \geq(1-t) \Phi^{-1}\left(\gamma_{n}(K)\right)+t \Phi^{-1}\left(\gamma_{n}(L)\right),
\end{equation}
i.e. $\Phi^{-1}\circ\gamma_n$ is concave, where $\Phi(x)=\gamma_{1}((-\infty, x))$. The inequality \eqref{e:Ehrhard_ineq} was first proven by Ehrhard for the case of two closed, convex sets \cite{EHR1,EHR2}.  Lata\l a \cite{Lat96} generalized Ehrhard's result to the case of an arbitrary Borel set $K$ and convex set $L$; the general case for two Borel sets of the Ehrhard's inequality was proven by Borell \cite{Bor03}. Since $\Phi$ is log-concave, the log-concavity of the Gaussian measure is strictly weaker than the Ehrhard inequality. Additionally, Kolesnikov and Livshyts showed that the Gaussian measure is $\frac{1}{2n}$ concave on $\conbod_0,$ the set of convex bodies containing the origin in their interior \cite{KL21}. That is, by restricting the admissible sets in the concavity equation, the concavity can improve.
\begin{corollary}[Berwald-type inequalities for the Gaussian Measure]
\label{cor:ber_gauss}
Let $f$ be a non-negative, concave function supported on $K\subset \R^n$. Then, we have the following:
\begin{enumerate}
    \item The function $$g_1(p)=\frac{1}{\Gamma(p+1)^{1/p}}\left(\frac{1}{\gamma_n(K)}\int_K f(x)^p d\gamma_n(x)\right)^{1/p}$$
    is strictly decreasing on $(-1,\infty);$
    \item The function $$g_2(p)=C(p,K)\left(\frac{1}{\gamma_n(K)}\int_K f(x)^p d\gamma_n(x)\right)^{1/p}$$
    is strictly decreasing on $(-1,\infty),$ where $C(p,K)=$
    $$
    \begin{cases}
    \left(\frac{p}{\gamma_n(K)}\int_0^\infty \Phi\left[\Phi^{-1}(\gamma_n(K))-t\right]t^{p-1} dt\right)^{-\frac{1}{p}} & \text {for } p>0 \\
    \left(\frac{p}{\gamma_n(K)}\int_{0}^{\infty}t^{p-1} (\Phi\left[\Phi^{-1}(\gamma_n(K)-t)\right]-\gamma_n(K))dt\right)^{-\frac{1}{p}} & \text {for } p\in (-1,0);
    \end{cases}
    $$
    \item and, if the maximum of $f$ is at the origin and $K\in\conbod_0,$ then the function $$g_3(p)=\left({{2n+p} \choose p}\frac{1}{\gamma_n(K)}\int_K f(x)^p d\gamma_n(x)\right)^{1/p}$$
    is decreasing on $(-1,\infty).$
\end{enumerate}
\end{corollary}

The equality condition for the third case of Corollary~\ref{cor:ber_gauss} can be deduced from Theorem~\ref{t:ber}, so we do not explicitly state it. If one further restricts the admissible sets, one can do even better. The Gardner-Zvavitch inequality states for symmetric $K,L\in\conbod_0$ and $t\in[0,1]$ that
\begin{equation}\label{e:gamma_gaussian}
    \gamma_n\left((1-t) K + t L\right)^{1/n}\geq (1-t)\gamma_n(K)^{1/n} + t \gamma_n(L)^{1/n},
\end{equation}
i.e. $\gamma_n$ is $1/n$-concave over the class of symmetric convex bodies. This inequality was first conjectured in \cite{GZ10} by Gardner and Zvavitch; a counterexample was shown in \cite{PT13} when $K$ and $L$ are not symmetric. Important progress was made in \cite{KL21}, which lead to the proof of the inequality \eqref{e:gamma_gaussian} by Eskenazis and Moschidis in \cite{EM21} for symmetric convex bodies. Recently, Cordero-Erausquin and Rotem \cite{CER23} extended this result to the class
\begin{equation}
\begin{split}
\mathcal{M}_n=\bigg\{&\text{Borel measures }\mu \text{ on }\R^n: d\mu(x)=e^{-w(|x|)}dx, w:[0,\infty)\to(-\infty,\infty]
\\
&\text{ is an increasing function such that } t\to w(e^t)\text{ is convex}\bigg\}.
\end{split}
\label{eq:measures_dario}
\end{equation}
That is, every measure $\mu\in\mathcal{M}_n$ is $1/n$-concave over the class of symmetric convex bodies. To show how rich this class is, $\mathcal{M}_n$ includes not only every rotationally invariant, log-concave measure (e.g. Gaussian), but also Cauchy-type measures. Combining these results, we obtain a Berwald-type inequality.
\begin{corollary}[Berwald-type inequality for rotationally invariant log-concave measures]
\label{cor:ber_rot_invar}
    Let $f$ be a non-negative, concave, even function supported on a symmetric $K\in \conbod_0.$ Let $\mu$ be a measure in $\mathcal{M}_n$ containing $K$ in its support. Then, for any $-1<p\leq q <\infty:$
    $$\left({{n+p} \choose p}\frac{1}{\mu(K)}\int_K f(x)^p d\mu(x)\right)^{1/p}\geq \left({{n+q} \choose q}\frac{1}{\mu(K)}\int_K f(x)^q d\mu(x)\right)^{1/q}.$$
\end{corollary}
We emphasize that the $(1/2n)$-concavity of the Gaussian measure on $\conbod_0$ shown in \cite{KL21} and the $1/n$-concavity of $\gamma_n$ and other measures from $\mathcal{M}_n$ over the class of symmetric convex bodies falls strictly outside the classification of $s$-concave measures by Borell. This paper is organized as follows. In Section~\ref{sec:ber}, we prove a version of Berwald's inequality for $F$-concave measures. In Section~\ref{sec:og_radial_bodies}, we discuss surface area measure, projection bodies, and radial mean bodies. Then, we apply our results to weighted generalizations of radial mean bodies. Along the way, we obtain more inequalities of Rogers and Shephard and of Zhang type. We would like to mention here that weighted extensions of concepts from the Brunn-Minkowski theory is a very rich field. This includes works on the surface area measure \cite{Ball93, Naz03,GAL15,GAL13,GAL21} and general measure extensions of the projection body of a convex body \cite{GAL19, LRZ22}. Recently, it has been shown that these developments, in particular the concavities for the Gaussian measure and Borell's classification, have led to a burgeoning \textit{weighted Brunn-Minkowski theory}, see \cite{KL23,FLMZ23_1,FLMZ23_2}.

\noindent \textbf{Acknowledgments} We would like to thank Artem Zvavitch for the helpful feedback throughout this work, and we also thank Matthieu Fradelizi for the discussion concerning Theorem~\ref{t:ber} when $p\in (-1,0)$. We would also like to thank Michael Roysdon for the discussions concerning this work, in particular the suggestion of Corollary~\ref{cor:perturb}. This work began during a visit to the Laboratoire d'Analyse et de Math{\'e}matiques Appliqu{\'e}es at Universit{\'e} Gustave Eiffel, France, from October to December 2021 and was continued during a visit to Tel Aviv University, Israel, in March and April 2022 - heartfelt thanks are extended, respectively, to Matthieu Fradelizi and Semyon Alesker.

\section{Generalizations of Berwald's Inequality}
\label{sec:ber}
In this section, we establish a generalization of Berwald's inequality. In what follows, for a finite Borel measure $\mu$ and a Borel set $K$ with positive $\mu$-measure, $\mu_K$ will denote the normalized probability on $K$ with respect to $\mu$, that is for measurable $A\subset\R^n:$ $\mu_K(A)=\frac{\mu(K\cap A)}{\mu(K)}.$ Notice that for every non-negative, measurable function $f$ on $K$ and $p>0$ such that $f\in L^p(\mu,K)$, one has the layer cake formula
$$\frac{1}{\mu(K)}\int_{K}f^{p}(x)d\mu(x)=p\int_0^\infty \mu_K(\left\{f\geq t\right\}) t^{p-1}dt$$
from the following use of Fubini's theorem:
\begin{align*}
    \frac{1}{\mu(K)}\int_{K}f^{p}(x)d\mu(x)&=\frac{p}{\mu(K)}\int_{K}\int_0^{f(x)}t^{p-1}dtd\mu(x)
    \\
    &=\frac{p}{\mu(K)}\int_{0}^\infty \mu(\left\{x\in K:f(x)\geq t\right\})t^{p-1}dt.
\end{align*}
Additionally, if $\mu$ is $F$-concave, with $F$ increasing and invertible, on a class $\mathcal{C}$ of convex sets, then for $K\in\mathcal{C}$ in the support of a concave function $f$, one has that the function given by $f_\mu(t)=\mu_K(\left\{f \geq t\right\})$ is $\tilde{F}$-concave, where $\tilde{F}(x)=F(\mu(K)x),$ as long as the level sets of $f$ belong to $\mathcal{C}.$ Indeed, since $f$ is concave, one has, for $\lambda\in[0,1]$ and $u,v\geq 0$, that
$$\{f\geq (1-\lambda)u+\lambda v\}\supset (1-\lambda)\{f\geq u\}+\lambda\{f\geq v\}.$$
Using the $F$-concavity of $\mu,$ this yields
$$F\left(\mu\left(\{f\geq(1-\lambda)u+\lambda v\}\right)\right) \geq (1-\lambda) F\left(\mu(\{f\geq u\})\right) + \lambda F\left(\mu(\{f\geq v\})\right).$$
Inserting the definition of $\tilde{F}$ and $f_\mu,$ this is precisely
$$\tilde{F}\circ f_{\mu}\left((1-\lambda)u + \lambda v\right) \geq (1-\lambda)\tilde{F}\circ f_{\mu}(u) +\lambda \tilde{F}\circ f_{\mu}(v).$$

\noindent Similarly one can check that if $\mu$ is $R$-concave, with $R$ decreasing and invertible, on a class $\mathcal{C}$ of convex sets, then for $K\in\mathcal{C}$ in the support of a concave function $f$, one then has that the function $f_\mu$ is $\tilde{R}$-convex, where $\tilde{R}(x)=R(\mu(K)x).$ That is, $\tilde{R}\circ f_{\mu}$ is a convex function on its support, as long as the level sets of $f$ belong to $\mathcal{C}.$

 We next need the appropriate layer cake formula for when $p<0.$ Notice that for every non-negative, measurable function $f$ on a Borel set $K$ and $p<0$ such that $f\in L^p(\mu,K)$ for a Borel measure $\mu$, one has the layer cake formula
$$\frac{1}{\mu(K)}\int_{K}f^{p}(x)d\mu(x)=p\int_0^\infty t^{p-1}(\mu_K(\left\{f\geq t\right\}) -1)dt$$
from the following use of Fubini's theorem:
\begin{align*}
    \frac{1}{\mu(K)}\int_{K}f^{p}(x)d\mu(x)&=-\frac{p}{\mu(K)}\int_{K}\int_{f(x)}^\infty t^{p-1}dtd\mu(x)
    \\
    &=\frac{p}{\mu(K)}\int_{0}^\infty t^{p-1} (\mu(\left\{x\in K:f(x)\geq t\right\})-\mu(K))dt.
\end{align*}

We now recall the analytic extension of the Gamma function. We start with the definition of $\Gamma(z)$ when the real part of $z$ is greater than zero: $$\Gamma(z)=\int_0^\infty t^{z-1}e^{-t}dt.$$ If the real part of $z$ is less than zero, then one uses analytic continuation to extend $\Gamma$ via the multiplicative property $\Gamma(z+1)=z\Gamma(z)$. Now, let us obtain the formula for $\Gamma(z)$ when the real part of $z$ is in $(-1,0)$. From the multiplicative property one can write
\begin{equation}\Gamma(z)=\frac{1}{z}\int_0^\infty t^{z}e^{-t}dt=\int_0^\infty t^{z-1}(e^{-t}-1)dt,\label{eq:gamma_neg}\end{equation}
where, for the second equality, integration by parts was performed and $e^{-t}$ was viewed as the derivative of $1-e^{-t},$ to maintain integrability. The fact that the layer cake formula looks similar to the formula for $\Gamma(z)$ when the real part of $z$ is between $-1$ and $0$ inspires the analytic continuation of Theorem~\ref{t:ber} to negative $p$. We will use the Mellin transformation, which was extended to $p\in (-1,0)$ in \cite{FLM20} for $s$-concave functions. We further generalize the Mellin transform here. 

The Mellin transform of a function $\psi$ such that $\text{supp}(\psi)\subseteq [0,B)$ is the analytic function for $p\in (-1,0)\cup (0,\infty)$ given by $\Mel{\psi}{p}=$
\begin{equation}
    \begin{cases}
    \int_0^B t^{p-1}(\psi(t)-\psi(0))dt +\frac{B^p}{p}\psi(0) & \text{for } p\in (-1,0),
    \\
    \int_0^B t^{p-1}\psi(t)dt & \text{for } p>0 \text{ such that } t^{p-1}\psi(t)\in L^1(\R).
    \end{cases}
    \label{eq:Mel}
\end{equation}
Following \cite{FLM20}, consider the function
\begin{equation}
    \psi_s(t)= \begin{cases}(1-t)^{1 / s} \chi_{[0,1]}(t) & \text {for } s>0, \\ 
    e^{-t} \chi_{(0,\infty)}(t) & \text {for } s=0, \\ 
    (1+t)^{1 / s} \chi_{(0,\infty)}(t) & \text {for } s<0.\end{cases}
\end{equation}
Then, for all $p>-1,$ one has $\Mel{\psi_s}{p}^{-1}=pC(p,s),$ where $C(p,s)$ is the constant defined in Corollary~\ref{cor:ber_s}, that is Berwald's inequality for $s$-concave measures; notice again that in the case when $s<0$, for $t^{p-1}(1+t)^{1/s}$ to be integrable, we must have that $p < -1/s.$

Motivated by this example, we need to define a function whose Mellin transform is related to the constant $C(p,\mu,K)$ from Theorem~\ref{t:ber}, and this definition will depend on the concavity of $\mu$. Recall that a function $\psi$ is $f$-concave for a monotonic function $f$ if $f\circ \psi$ is either concave (if $f$ is increasing) or convex (if $f$ is decreasing). Similarly, $\psi$ is $f$-affine if $f\circ \psi$ is an affine function. We will have three different restrictions on the function $f$, matching those in Theorem~\ref{t:ber} (and the notation as well). First, fix some $A>0$. Then, we will consider the case when $f\in\{F,Q,R\},$ where $F$ represents those functions $F:[0,A]\to [0,\infty)$ that are continuous, increasing and invertible; $Q$ represents those functions $Q:(0,A]\to (-\infty,\infty)$ that continuous, increasing and invertible; and $R$ represents those functions $R:(0,A]\to (0,\infty)$ that are continuous, decreasing and invertible. We next define
\begin{equation}
    \psi_{f,A}(t)= \begin{cases}F^{-1}(F(A)(1-t)) \chi_{[0,1]}(t) & \text {if } f=F, \\ 
    Q^{-1}(Q(A)-t)\chi_{(0,\infty)}(t) & \text {if } f=Q, \\ 
    R^{-1}(R(A)(1+t)) \chi_{(0,\infty)}(t) & \text {if} f=R.\end{cases}
    \label{F_affine}
\end{equation}
Notice that, if $A=\mu(K),$ then $\Mel{\psi_{f,\mu(K)}}{p}^{-1}=\frac{p}{\mu(K)}C(p,\mu,K)^p$ if $p\in (-1,0)$, and this also holds for any $p>0$ such that $t^{p-1}\psi_{f,\mu(K)}$ is integrable. 

We will now work towards the proof of Theorem~\ref{t:ber}. Let $\psi$ be a non-negative function such that $\psi(0)=A>0.$ Then, for $p\in (-1,0)\cup (0,p_1),$ set
\begin{equation}
    \Omega_{f,\psi}(p)=\frac{\Mel{\psi}{p}}{\Mel{\psi_{f,A}}{p}},
    \label{Omega_level}
\end{equation}
where $\Omega_{f,\psi}(0)=1$ and $p_1$ is defined implicitly by $p_1=\sup\{p<\infty:\Omega_{f,\psi}(p)<\infty\}.$ Next, set for $p\in (-1,0)\cup(0,p_1)$
\begin{equation}
    G_{\psi}(p)=\left(\Omega_{f,\psi}(p)\right)^{1/p}
    \label{eq:milman_psi}
\end{equation}
and $G_{\psi}(0)=\exp(\log(\Omega_{f,\psi})^\prime(0)).$
\begin{lemma}[The Mellin-Berwald Inequality]
\label{l:mel_ber}
Let $\psi:[0,\infty)\to[0.\infty)$ be an integrable, $f$-concave function, $f\in \{F,Q,R\}$ (elaborated above \eqref{F_affine}). Suppose that $\psi$ is right differentiable at the origin. Next, set $p_0=\inf\{p>-1:\Omega_{f,\psi}(p)>0\},$ where $ \Omega_{f,\psi}(p)$ is defined via \eqref{Omega_level}. Then,
\begin{enumerate}
    \item $p_0 \in [-1,0)$ and if $\psi$ is non-increasing then $p_0=-1.$
    \item $\Omega_{f,\psi}(p)>0$ for every $p\in (p_0,p_1).$ Thus, $G_{\psi}(p),$ defined via \eqref{eq:milman_psi}, is well-defined and analytic on $(p_0,p_1).$
    \item $G_{\psi}(p)$ is non-increasing on $(p_0,p_1).$
    \item If there exists $r,q\in (p_0,p_1)$ such that $G_{\psi}(r)=G_{\psi}(q),$ then $G_{\psi}(p)$ is constant on $(p_0,p_1).$ Furthermore, $G_{\psi}(p)$ is constant on $(p_0,p_1)$ if, and only if, $\psi(t)=\psi_{f,A}(\frac{t}{\alpha})$ for some $\alpha >0,$ in which case $G_{\psi}(p)=\alpha.$
\end{enumerate}
\end{lemma}
\begin{proof}
    From the fact that $\Omega_{f,\psi}(0)=\psi(0)=1>0,$ one immediately has that $p_0\in [-1,0).$ Notice that $\Mel{\psi_{f,A}}{p} <0$ for $p\in (-1,0).$ If $\psi$ is non-increasing, then from \eqref{eq:Mel} one obtains that $\Mel{\psi}{p}<0$ as well. Thus, $\Omega_{f,\psi}(p)=\Mel{\psi}{p}/\Mel{\psi_{f,A}}{p}>0$ for all $p\in (-1,0),$ and thus $p_0=-1.$
    
    For the second statement, clearly $\Omega_{f,\psi}(p)>0$ for $p\in [0,p_1].$ So, fix some $q\in (p_0,0)$ such that $\Omega_{f,\psi}(q)>0.$ Then, $G_{\psi}(q)=\left(\Omega_{f,\psi}(q)\right)^{1/q}>0.$ Define the function $z(t)=\psi_{f,A}(t/ G_{\psi}(q)).$ Notice that $z(0)=\psi_{f,A}(0)=A$ and, by performing a variable substitution, $\Mel{z}{p}=(G_{\psi}(q))^p\Mel{\psi_{f,A}}{p}$ via \eqref{eq:Mel} for every $p\in (-1,0)\cup(0,p_1).$ In particular, for $p=q.$ From the definition of $G_\psi (q),$ we then obtain that $\Mel{z}{q}=(G_{\psi}(q))^q\Mel{\psi_{f,A}}{q}=\Mel{\psi}{q}.$ Thus, from \eqref{eq:Mel}, one obtains
    $$0=\Mel{\psi}{q}-\Mel{z}{q}=\int_0^\infty t^{q-1}(\psi(t)-z(t))dt.$$
    Consequently, the function $\psi(t)-z(t)$ changes signs at least once. But actually, this function changes sign exactly once. Indeed, let $t_0$ be the smallest positive value such that $\psi(t_0)=z(t_0).$ Then, $f\circ\psi(t_0)=f\circ z(t_0).$ Now, $f\circ z$ is affine. If $f\in\{F,Q\},$ then $f\circ \psi$ is concave and the slope of $f\circ z$ is negative. Since $\psi(0)=z(0)=A,$ one has that $f\circ \psi (t) \geq f\circ z(t)$ on $[0,t_0].$ From the concavity, we must then have that $f\circ \psi (t) \leq f\circ z(t)$ on $[t_0,\infty).$ Similarly, if $f=R,$ then $f\circ \psi$ is convex and the slope of $f\circ z$ is positive. Hence, $f\circ \psi (t) \leq f\circ z(t)$ on $[0,t_0]$ and $f\circ \psi (t) \geq f\circ z(t)$ on $[t_0,\infty).$ Taking inverses, we obtain in either case that $\psi (t) \geq z(t)$ on $[0,t_0]$ and $\psi (t) \leq z(t)$ on $[t_0,\infty).$
    
    Next, define $$g(t)=\int_t^\infty u^{q-1}(\psi(u)-z(u))du.$$ Clearly, $g(0)=g(\infty)=0.$ One has $g^\prime (t)=-t^{q-t}(\psi(t)-z(t)).$ Thus, $g$ is non-increasing on $[0,t_0]$ and non-decreasing on $[t_0,\infty).$ Hence $g(t)\leq 0$ for all $t\in [0,\infty).$ Next, pick $r \in (q,0).$ From integration by parts, one obtains
    $$\Mel{\psi}{r}-\Mel{z}{r}=\int_0^\infty t^{r-q}t^{q-1}(\psi(t)-z(t))dt=(r-q)\int_0^\infty t^{r-q-1}g(t)dt \leq 0.$$
    Hence,
    $$\Mel{\psi}{r}\leq \Mel{z}{r}=(G_{\psi}(q))^r\Mel{\psi_{f,A}}{r}<0.$$
    We deduce that
    \begin{equation}\Omega_{f,\psi}(r)=\frac{\Mel{\psi}{r}}{\Mel{\psi_{f,A}}{r}} \geq (G_{\psi}(q))^r >0
    \label{eq:bootstrap}
    \end{equation}
    for every $r\in (q,0).$ Sending $q\to p_0,$ we obtain $\Omega_{f,\psi}(p)>0$ for every $p\in (p_0,0)$ and thus for $p\in(p_0,p_1).$ One immediately obtains that $G_\psi(p)$ is well-defined and analytic on $(p_0,p_1).$ Finally, taking the $r$th root of \eqref{eq:bootstrap} yields for $p_0<q<r<0$ that
    $$G_\psi(r) = (\Omega_{f,\psi}(r))^{1/r}\leq G_\psi(q),$$
    i.e. $G_\psi(p)$ is non-increasing on $(p_0,0).$ Suppose there exists an $r\in (q,0)$ such that $G_\psi(q)=G_\psi(r).$ Then, there is equality in \eqref{eq:bootstrap}. But this yields $g(t)=0$ for almost all $t.$ We take a moment to notice that this then yields $G_\psi(q)=G_\psi(r)$ for every $q,r\in (p_0,0).$ Anyway, since $g(t)=0$ for almost all $t$, we have $\psi(t)=z(t)$ for almost all $t$. Hence, the concave function $f\circ \psi(t)$ equals the affine function $f\circ z(t)$ for almost all $t$ and thus for all $t$. Consequently, $\psi(t)\equiv z(t)=\psi_{f,A}(t/ G_{\psi}(q)).$ Conversely, suppose that $\psi(t)=\psi_{f,A}(t/ \alpha)$ for some $\alpha>0.$ Then, direct substitution yields $G_{\psi}(p)=\alpha$ on $(p_0,0).$ Notice that $\Mel{z}{q}=(G_{\psi}(q))^q\Mel{\psi_{f,A}}{q}=\Mel{\psi}{q}$ is also true for any $q \in (0,p_1).$ Consequently, by picking any $r\in (q,p_1),$ we repeat the above arguments and deduce again that $$\Mel{\psi}{r}\leq \Mel{z}{r}=(G_{\psi}(q))^r\Mel{\psi_{f,A}}{r}.$$ This time, however, $\Mel{\psi_{f,A}}{r} >0.$ Consequently, this immediately implies that $$G_\psi(r) = (\Omega_{f,\psi}(r))^{1/r}\leq G_\psi(q)$$
    for every $0<q\leq r < p_1.$ This establishes that $G_\psi(p)$ is non-increasing on $(0,p_1)$ as well. The argument for the equality conditions is the same.
\end{proof}
\begin{proof}[Proof of Theorem~\ref{t:ber}]
    Let $w$ be the concavity of our measure $\mu.$ Next, let $\psi(t)=\mu(\left\{x\in K:f(x) \geq t\right\}).$ Notice this $\psi$ is non-increasing, and thus $p_0$ from the statement of Lemma~\ref{l:mel_ber} is $-1$. Then, for $p\in(-1,0):$
    \begin{align*}\Omega_{w,\mu(K),\psi}(p)&=\frac{\Mel{\psi}{p}}{\Mel{\psi_{w,\mu(K)}}{p}}
    \\
    &=C^p(p,\mu,K)\frac{p}{\mu(K)}\int_{0}^\infty t^{p-1} (\mu(\left\{x\in K:f(x)\geq t\right\})-\mu(K))dt
    \\
    &=C^p(p,\mu,K)\frac{1}{\mu(K)}\int_K f^p(x)d\mu(x)\end{align*}
    via the layer cake formula for $p\in (-1,0)$; similar computations yield the case for $p>0,$ and $p=0$ follows from limits. 
    Thus, we obtain from Lemma~\ref{l:mel_ber}, Item 3, that the function
    $$G_{\psi}(p)=C(p,\mu,K)\left(\frac{1}{\mu(K)}\int_K f^p(x)d\mu(x)\right)^{1/p}$$
    is non-increasing for $p\in(-1,p_{\text{max}}).$ Furthermore, $G_{\psi}(p) \equiv \alpha>0$, if, and only if, $$\mu(\left\{x\in K:f(x)\geq t\right\})=\psi(t)=\psi_{w,\mu(K)}(t/ \alpha).$$ 
    
    We now insert the appropriate $\psi_{w,\mu(K)}$, starting with the case $w=F$. This is precisely
\begin{equation}\alpha t=1-\frac{F(\mu(\{f\geq t\}))}{F(\mu(K))} \longleftrightarrow \mu(\{f\geq t\})=F^{-1}\left[F(\mu(K))\left(1-\alpha t\right)\right].
\label{eq:level_sets_eq}
\end{equation}
We then evaluate the above at $t=\|f\|_{\infty},$ to obtain $$\alpha = \left(1-\frac{F(0)}{F(\mu(K))}\right)/\|f\|_{\infty}.$$ On the other hand, we also know that, for all $p\in (0,\infty)$ we have 
\begin{align*}\alpha^p=\frac{\int_0^1 F^{-1}\left[F(\mu(K))(1-t)\right]t^{p-1} dt}{\int_0^{\|f\|_\infty} \mu(\{f\geq t\})t^{p-1}dt}.\end{align*}
Inserting the formula for $\alpha$ and the formula of $\mu(\{f\geq t\})$ from \eqref{eq:level_sets_eq}, we obtain
$$\frac{(1-\frac{F(0)}{F(\mu(K))})^p}{\|f\|_\infty^p}=\frac{\int_0^1 F^{-1}\left[F(\mu(K))(1-t)\right]t^{p-1} dt}{\int_0^{\|f\|_\infty} F^{-1}\left[F(\mu(K))\left(1-\frac{(1-\frac{F(0)}{F(\mu(K))})}{\|f\|_{\infty}}t\right)\right]t^{p-1}dt}.$$
By performing a variable substitution in the denominator, we obtain that
$$1=\frac{\int_0^1 F^{-1}\left[F(\mu(K))(1-t)\right]t^{p-1} dt}{\int_0^{(1-\frac{F(0)}{F(\mu(K))})} F^{-1}\left[F(\mu(K))\left(1-t\right)\right]t^{p-1}dt}.$$
Therefore, we have $(1-\frac{F(0)}{F(\mu(K))})=1,$ which means $F(0)=0.$

Next, we show that equality never occurs for when $w=Q$, and the case $w=R$ is similar. From integrability, we have that $Q^{-1}(-\infty)=0,$ or $Q(0)=-\infty$ (where these are understood as limits from the left and the right, respectively). On the other hand, 
we have shown equality implies $$\alpha t=Q(\mu(K))-Q(\mu(K)f_{\mu}(t)).$$ Evaluating again at $t=\|f\|_\infty$ yields $\alpha\|f\|_\infty=Q(\mu(K))-Q(0),$ which would imply that $|Q(0)| < \infty.$    
\end{proof}

\begin{proof}[Proof of Corollary~\ref{cor:ber_s}]
We have that $\mu$ is $s$-concave on the level sets of $f$, and thus the proof is a direct application of Theorem~\ref{t:ber}; in the first case, the coefficients become a beta function and in the second case they become a gamma function. As for the third case, a bit more work is required. We will show the case when $p\in (0,-1/s);$ the case when $p\in (-1,0)$ is exactly the same (using the analytic continuation of the Beta function), and then the case $p=0$ follows from limits. Inserting $R(x)=x^s, s<0$ yields
$$C(p,s)=\left(p\int_0^\infty \left(1+t\right)^{1/s}t^{p-1} dt\right)^{-1}.$$
Focus on the function $q(t)=\left(1+t\right)^{1/s}t^{p-1}.$ For this function to be integrable near zero, we require $-1<p-1,$ and, for the integrability near infinity, we require $\frac{1}{s}+p-1<-1.$ Thus, $p\in (0,-1/s).$ We will now manipulate $C(p,s)$ to obtain a more familiar formula. Consider the variable substitution given by $t=\frac{z}{1-z}.$ Writing $z$ as a function of $t,$ this becomes
$$z=1-\frac{1}{1+t}\quad \longrightarrow \quad z^\prime (t)=\frac{1}{(1+t)^2}.$$ 

\noindent As $t\to 0^+, z\to 0^+,$ and as $t\to \infty, z\to 1^-.$ We then obtain that
\begin{align*}C(p,s)&=\left(p\int_0^1 \left(1-z\right)^{-(p+1/s)-1}z^{p-1} dz\right)^{-1}=\frac{\Gamma\left(-\frac{1}{s}\right)}{p\Gamma\left(p\right)\Gamma\left(-p-\frac{1}{s}\right)}
\\
&=s\left(p+\frac{1}{s}\right)\frac{\Gamma\left(1-\frac{1}{s}\right)}{\Gamma\left(1+p\right)\Gamma\left(1-p-\frac{1}{s}\right)},\end{align*}
which equals our claim.
\end{proof}

\begin{proof}[Proof of Theorem~\ref{t:sberh}]
From the assumptions on the measure $\mu$, we obtain that $d\mu(x)=\phi(x)dx$ for some $p=s/(1-ns)$-concave function $\phi.$ Furthermore, $\phi$ is $(1/s)-n$ homogeneous. Observe that Corollary~\ref{cor:ber_s} yields the inequality; all that remains to show is the equality conditions. By hypothesis, the maximum of the function $f$ is obtained at the origin. Equality conditions of Corollary~\ref{cor:ber_s} imply that
$$\|f\|^{1/s}_\infty=\frac{\int_0^{\|f\|_\infty} \mu_K(\{f\geq t\})t^{1/s-1}dt}{\int_0^1 (1-t)^{1/s}t^{1/s-1} dt}.$$ Using \eqref{eq:star_form}, this implies that 
\begin{align*}\int_K f^{1/s}(x)d\mu(x)&=\frac{\mu(K)}{s}\int_0^1 [\|f\|_\infty(1-t)]^{1/s} dt
\\
&=\int_{\s^{n-1}}\phi(\theta)\rho_K(\theta)^{1/s}d\theta\int_0^1 [\|f\|_\infty(1-t)]^{1/s}t^{1/s-1}dt.\end{align*}
Using Fubini's theorem, a variable substitution $t\to t/\rho_K(\theta)$ and the homogeneity of $\phi$ yields
\begin{align*}\int_K f^{1/s}(x)d\mu(x) &=\int_{\s^{n-1}}\int_0^{\rho_K(\theta)} \left[\|f\|_\infty\left(1-\frac{t}{\rho_K(\theta)}\right)\right]^{1/s}t^{n-1}\phi(t\theta)dtd\theta
\\
&=\int_K\left[\|f\|_\infty\left(1-\frac{1}{\rho_K(x)}\right)\right]^{1/s}dx.\end{align*}
One has from \eqref{eq:supporting_line_bound} that a concave function $f$ supported on $K\in\conbod_0$ whose maximum is at the origin satisfies
$$f^{1/s}(x) \geq \left[\|f\|_\infty\left(1-\frac{1}{\rho_K(x)}\right)\right]^{1/s}, \; x\in K\setminus\{0\}. $$
By the above integral, we have equality.
\end{proof}
 We next obtain an interesting result by perturbing Theorem~\ref{t:sberh}, inspired by the standard proof (see e.g. \cite{gardner_book}) of Minkowski's first inequality by perturbing the Brunn-Minkowski inequality.

\begin{corollary}
 Let $\mu$ be a Radon measure that is $s$-concave, $1/s$-homogeneous, $s\in (0,1/n]$, and suppose that $\ell_K$ is given by \eqref{eq:lk} for some $K\in\conbod$.  Let $\psi$ be a concave function supported on $K$, and suppose $0<p\leq q < \infty.$  Then, one has

$${{\frac{1}{s}+p}\choose \frac{1}{s}} \int_K \ell_K^p(x)\left(\frac{\psi(x)}{\ell_K(x)}\right)d\mu(x) \geq {{\frac{1}{s}+q}\choose \frac{1}{s}} \int_K \ell_K^q(x)\left(\frac{\psi(x)}{\ell_K(x)}\right)d\mu(x).$$
\label{cor:perturb}
\end{corollary}
\begin{proof}
Let $z_K(t,x)$ be a concave perturbation of $\ell_K$ by $\psi$, i.e. $\delta>0$ is picked small enough so that $z_K(t,x)=\ell_K(x)+t\psi(x)$ is concave with maximum at the origin for all $x\in K$ and $|t|<\delta.$ Next, consider the function given by, for $0< p\leq q$
\begin{align*}B_K(t)=&\left({{\frac{1}{s}+p}\choose \frac{1}{s}}  \frac{1}{\mu(K)}\int_K z_K(x,t) d\mu(x)\right)^{1/p}
\\
&-\left({{\frac{1}{s}+q}\choose \frac{1}{s}} \frac{1}{\mu(K)}\int_K z_K(x,t) d\mu(x)\right)^{1/q},\end{align*}
from Berwald's inequality in Theorem~\ref{t:sberh}, this function is greater than or equal to zero for all $|t|<\delta,$ and equals zero when $t=0.$ Hence, the derivative of this function is non-negative at $t=0$. By taking the derivative of $B_K(t)$ in the variable t, evaluating at $t=0$, and setting this computation be greater than or equal to zero, one immediately obtains the result.
\end{proof}

We now prove the corollaries for the Gaussian measure and rotational invariant log-concave measures.
\begin{proof}[Proof of Corollary~\ref{cor:ber_gauss}]
From Borell's classification, the Gaussian measure is log-concave, and thus one can use the second case of Corollary~\ref{cor:ber_s} for the first inequality. For the second inequality, the function $\Phi^{-1}$ behaves logarithmically, that is one can apply the second case of Theorem~\ref{t:ber}. Finally, for the third inequality, note that if $f$ is a concave function supported on some $K\in\conbod_0$ with maximum at the origin, then the level sets of $f$ are also in $\conbod_0,$ and thus one can apply the $\frac{1}{2n}$-concavity of the Gaussian measure over $\conbod_0$ and use the first case of Corollary~\ref{cor:ber_s}.
\end{proof}
\begin{proof}[Proof of Corollary~\ref{cor:ber_rot_invar}]
Notice that if $f$ is an even, concave function supported on a symmetric $K\in\conbod_0$, then the maximum of $f$ is at the origin (for every $x\in K, -x\in K$ and so $f(0)=f(\frac{1}{2}x+\frac{1}{2}(-x))\geq \frac{1}{2}f(x)+\frac{1}{2}f(-x)=f(x)$) and the level sets of $f$ are all symmetric convex bodies. Thus, the result follows from the $1/n$-concavity of measures in $\mathcal{M}_n$.
\end{proof}

\subsection{Applications}
We conclude this section by showing a few applications. The first example uses that the support of $f$ in Theorem~\ref{t:ber} need not be compact.
\begin{theorem}
\label{t:log_norm}
Let $\theta\in\s^{n-1}$. Denote $H=\theta^\perp$ and $H_+=\{x\in\R^n:\langle x,\theta \rangle >0\}.$ Denote 
$$\langle x,\theta \rangle_+=\langle x,\theta \rangle\chi_{H_+}(x)=
\begin{cases}
\langle x,\theta \rangle \; &\text{if} \; \langle x,\theta \rangle>0,
\\
0 &\text{otherwise.}
\end{cases}
$$
Then, for every Borel measure $\mu$ finite on $H_+$ with one of the following concavity conditions on subsets of $H_+$:
\begin{enumerate}
    \item If $\mu$ is $F$-concave, where $F:[0,\mu(H_+)]\to [0,\infty)$ is an increasing and invertible function one has
    $$\frac{\left(\int_{\R^n} \langle x,\theta \rangle_+^qd\mu(x)\right)^{1/q}}{\left(\int_{\R^n}\langle x,\theta \rangle_+^pd\mu(x)\right)^{1/p}}\leq \frac{\left(q\int_0^1 \left(F^{-1}\left[F(\mu(H_+))(1-t)\right]-\mu(H_+)\right)t^{q-1} dt+\mu(H_+)\right)^{1/q}}{\left(p\int_0^1 \left(F^{-1}\left[F(\mu(H_+))(1-t)\right]-\mu(H_+)\right)t^{p-1} dt+\mu(H_+)\right)^{1/p}}$$
    for every $-1<p\leq q <\infty$ where the integrals exist.  In particular, if $F(x)=x^s, s>0,$ one obtains
$$\left(\int_{\R^n} \langle x,\theta \rangle_+^qd\mu(x)\right)^{1/q}\leq\mu(H_+)^{\frac{1}{q}-\frac{1}{p}} \frac{{{\frac{1}{s}+p}\choose p}^{1/p}}{{{\frac{1}{s}+q}\choose q}^{1/q}}\left(\int_{\R^n}\langle x,\theta \rangle_+^pd\mu(x)\right)^{1/p}.$$

\item If $\mu$ is $Q$-concave, where $Q:(0,\mu(H_+)]\to (-\infty,\infty)$ is an increasing and invertible function one has
$$\frac{\left(\int_{\R^n} \langle x,\theta \rangle_+^qd\mu(x)\right)^{1/q}}{\left(\int_{\R^n}\langle x,\theta \rangle_+^pd\mu(x)\right)^{1/p}}\leq \frac{\left(q\int_0^\infty Q^{-1}\left[Q(\mu(H_+))-t\right]t^{q-1} dt\right)^{1/q}}{\left(p\int_0^\infty Q^{-1}\left[Q(\mu(H_+))-t\right]t^{p-1} dt\right)^{1/p}}$$
for every $0<p\leq q <\infty$ where the integrals exist; the case for $-1<p\leq q < \infty$ can be deduced. For the Gaussian measure especially, one can set $Q=\Phi^{-1}$ and obtain
$$\frac{\left(\int_{\R^n} \langle x,\theta \rangle_+^qd\gamma_n(x)\right)^{1/q}}{\left(\int_{\R^n}\langle x,\theta \rangle_+^pd\gamma_n(x)\right)^{1/p}}\leq \frac{\left(q\int_0^\infty \Phi\left[\Phi^{-1}(\gamma_n(H_+))-t\right]t^{q-1} dt\right)^{1/q}}{\left(p\int_0^\infty \Phi\left[\Phi^{-1}(\gamma_n(H_+))-t\right]t^{p-1} dt\right)^{1/p}}.$$

\noindent If $Q(x)=\log(x)$ one obtains for every $-1<p\leq q<\infty$ that
$$\left(\int_{\R^n} \langle x,\theta \rangle_+^qd\mu(x)\right)^{1/q}\leq\mu(H_+)^{\frac{1}{q}-\frac{1}{p}} \frac{{\Gamma\left(q+1\right)}^{1/q}}{{\Gamma\left(p+1\right)}^{1/p}}\left(\int_{\R^n}\langle x,\theta \rangle_+^pd\mu(x)\right)^{1/p}.$$

\item If $\mu$ is $R$-concave, where $R:(0,\mu(H_+)]\to (0,\infty)$ is a decreasing and invertible function one has
$$\frac{\left(\int_{\R^n} \langle x,\theta \rangle_+^qd\mu(x)\right)^{1/q}}{\left(\int_{\R^n}\langle x,\theta \rangle_+^pd\mu(x)\right)^{1/p}}\leq \frac{\left(q\int_0^\infty R^{-1}\left[R(\mu(H_+))(1+t)\right]t^{q-1} dt\right)^{1/q}}{\left(p\int_0^\infty R^{-1}\left[R(\mu(H_+))(1+t)\right]t^{p-1} dt\right)^{1/p}}$$
for every $0<p\leq q <\infty$ where the integrals exist; the case for $-1<p\leq q < \infty$ can be deduced. In particular, if $R(x)=x^s, s<0,$ and $-1<p\leq q < -1/s,$ one obtains
$$\left(\int_{\R^n} \langle x,\theta \rangle_+^qd\mu(x)\right)^{1/q}\leq\mu(H_+)^{\frac{1}{q}-\frac{1}{p}} \frac{\left(s\left(p+\frac{1}{s}\right){{-\frac{1}{s}}\choose p}\right)^{1/p}}{\left(s\left(q+\frac{1}{s}\right){{-\frac{1}{s}}\choose q}\right)^{1/q}}\left(\int_{\R^n}\langle x,\theta \rangle_+^pd\mu(x)\right)^{1/p}.$$

\end{enumerate}

\end{theorem}

Finally, let $\mu$ be a Borel measure finite on some convex $K\subset \R^n.$ Suppose $\mu$ is either $F,Q$ or $R$ concave, where the functions $F,Q$ and $R$ are as given in Theorem~\ref{t:ber}. Next, consider a non-negative function $f$ so that $f^\beta$ is bounded and concave on $K$ for some $\beta>0$. Inserting $f^\beta,$ into Theorem~\ref{t:ber} and picking appropriate choices of $p$ and $q,$ we obtain that for every $q\geq 1$ one has
\begin{equation}\left(\int_K f(x)^{q} d\mu(x)\right)^{1/q} \leq \mu(K)^{\frac{1-q}{q}}\left(\frac{C(\frac{1}{\beta},\mu,K)}{C(\frac{q}{\beta},\mu,K)}\right)^{\frac{1}{\beta}}\int_K f(x) d\mu(x),\label{eq:ber_norm_relate}\end{equation}
up to possible restrictions on admissible $\beta$ and $q$ so that all constants exist. In words, we have bounded the $L^q(K,\mu)$ norm of a bounded, non-negative, $\beta$-concave function $f$ by its $L^1(K,\mu)$ norm when $\mu$ is either $F, Q$ or $R$-concave. Examples of interest are when $\mu$ is $s$-concave. We obtain for a $s$-concave measure $\mu$ and $q\geq 1$:
\begin{enumerate}
    \item When $s>0$: \begin{equation*}\left(\int_K f(x)^{q} d\mu(x)\right)^{1/q}\leq\frac{{{\frac{1}{s}+\frac{1}{\beta}}\choose \frac{1}{\beta}}}{\mu(K)}\left(\frac{\mu(K)}{{{\frac{1}{s}+\frac{q}{\beta}}\choose \frac{q}{\beta}}}\right)^{1/q}\int_K f(x) d\mu(x).\label{eq:norm_relate_pos}\end{equation*}
    \item When $s=0$: \begin{equation*}\left(\int_K f(x)^{q} d\mu(x)\right)^{1/q}\leq\frac{\Gamma(1+\frac{1}{\beta})}{\mu(K)}\left(\frac{\mu(K)}{\Gamma(1+\frac{q}{\beta})}\right)^{1/q}\int_K f(x) d\mu(x).\label{eq:norm_relate_log}\end{equation*}
    \item When $s < 0,$ $\beta > -s$ and $q \in [1,-\frac{\beta}{s}):$
    \begin{equation*}\left(\int_K f(x)^{q} d\mu(x)\right)^{1/q}\leq\frac{s\left(q+\frac{1}{s}\right){{-\frac{1}{s}}\choose q}}{\mu(K)}\left(\frac{\mu(K)}{s\left(\frac{q}{\beta}+\frac{1}{s}\right){{-\frac{1}{s}}\choose \frac{q}{\beta}}}\right)^{1/q}\int_K f(x) d\mu(x).\label{eq:norm_relate_neg}\end{equation*}
    \end{enumerate}
    We also highlight the following examples for the Gaussian measure.
    \begin{enumerate}
    \item \begin{equation*}\left(\int_K f(x)^{q} d\gamma_n(x)\right)^{1/q}\leq\beta^{\frac{q-1}{q}}\frac{\left(q\int_0^\infty \Phi\left[\Phi^{-1}(\gamma_n(K))-t\right]t^{q-1} dt\right)^{1/q}}{\int_0^\infty \Phi\left[\Phi^{-1}(\gamma_n(K))-t\right]t^{p-1} dt}\int_K f(x) d\gamma_n(x).\end{equation*}
    
    \item If $K\in\conbod_0$ and the maximum of $f^\beta$ is obtained at the origin:
    \begin{equation*}\left(\int_K f(x)^{q} d\gamma_n(x)\right)^{1/q}\leq\frac{{{\frac{1}{2n}+\frac{1}{\beta}}\choose \frac{1}{\beta}}}{\gamma_n(K)}\left(\frac{\gamma_n(K)}{{{\frac{1}{2n}+\frac{q}{\beta}}\choose \frac{q}{\beta}}}\right)^{1/q}\int_K f(x) d\gamma_n(x).\end{equation*}
    
    \item Let $\mu$ be a measure in $\mathcal{M}_n$. If $K\in\conbod_0$ is symmetric, and $f^\beta$ is even:
    \begin{equation*}\left(\int_K f(x)^{q} d\mu(x)\right)^{1/q}\leq\frac{{{\frac{1}{n}+\frac{1}{\beta}}\choose \frac{1}{\beta}}}{\mu(K)}\left(\frac{\mu(K)}{{{\frac{1}{n}+\frac{q}{\beta}}\choose \frac{q}{\beta}}}\right)^{1/q}\int_K f(x) d\mu(x).\end{equation*}
\end{enumerate}
\noindent To see how \eqref{eq:ber_norm_relate} yields results for the relative entropy of two measures with concavity, based on the work by Bobkov and Madiman \cite{BM12} for Boltzmann-Shannon entropy, see \cite{BFLM}.

\section{Applications to Convex Geometry}
\label{sec:og_radial_bodies}
\subsection{Weighted Radial Mean Bodies}
Throughout this section, we write $\lambda_n$ for the Lebesgue measure on $\R^n$. 

One of our motivations for generalizing Berwald's inequality is to study generalizations of the projection body and radial mean bodies of a convex body. We first introduce weighted radial mean bodies. For a Borel measure $\mu$ finite on a Borel set $K$ in its support, the $p$th mean of a non-negative $f\in L^p(K,\mu)$ is
\begin{equation}
    \label{pth_mean}
    M_{p,\mu} f=\left(\frac{1}{\mu(K)}\int_K f(x)^p d\mu(x) \right)^{\frac{1}{p}}.
\end{equation}
Jensen's inequality states that $M_{\mu,p} f \leq M_{\mu,q} f$ for $p\leq q.$ From continuity, one has
$\lim_{p\to\infty} M_{p,\mu} f=\text{ess} \sup_{x\in K}f(x),$
and
$$\lim _{p \rightarrow 0} M_{p,\mu} f=\exp \left(\frac{1}{\mu(K)} \int_{K} \log f(x) d \mu(x)\right).$$ 
Recall that for a star body $K$, $\rho_K(x,\theta):=\rho_{K-x}(-\theta)=\sup\{\lambda:-\lambda\theta\in K-x\}$.
\label{sec:mupth}
\begin{definition}
Let $\mu$ be a Borel measure on $\R^n$ and $K$ a convex body contained in the support of $\mu$. Then, the $\mu$-weighted $p$th radial mean body of $K$, denoted $R_{p,\mu} K,$ is the star body whose radial function is given, for $p\in(-1,\infty)$ and $\theta\in\s^{n-1},$  as
$$\rho_{R_{p,\mu} K}(\theta)=\left(\frac{1}{\mu(K)}\int_K \rho_K(x,\theta)^p d\mu(x) \right)^{\frac{1}{p}}.$$
\end{definition}
\noindent We will show why $R_{p,\mu} K$ exists for $p\in (-1,0)$ later in this section. The usual radial mean bodies $R_p K$, first defined by Gardner and Zhang \cite{GZ98}, are precisely the bodies $R_{p,\lambda_n} K$ in our notation. Sending $p$ to $\infty$ or $0$, we obtain the limiting bodies $R_{\infty, \mu} K$ and $R_{0, \mu} K$, given in terms of their radial functions by 
\begin{align*}
    \rho_{R_{\infty,\mu} K}(\theta) &= \max_{x\in K}\rho_K(x,\theta)=\rho_{DK}(\theta), \quad \text{and} \\
\rho_{R_{0, \mu} K}(\theta) &= \exp \left(\frac{1}{\mu(K)} \int_{K} \log \rho_{K}(x, \theta) d \mu(x)\right),
\end{align*}
where $DK$ is the difference body of $K$, given by 
	\begin{equation}\label{e:differnce}
	    DK=\{x:K\cap(K+x)\neq \emptyset\}=K+(-K).
	\end{equation}

A natural question is how $R_{p,\mu} K$ behaves under linear transformations. We introduce the following notation: for a Borel measure $\mu$ and $T \in SL_n$, we denote by $\mu^T$ the pushforward of $\mu$ by $T^{-1}$; note that if $\mu$ has density $\phi$ then $\mu^T$ has density $\phi \circ T$, and $\mu^T(A)=\mu(TA)$ for a Borel set $A$.
\begin{proposition}
Let $\mu$ be a Borel measure finite on a convex body $K \subset \supp \mu$. Then, for $T \in SL_n$ and $p>-1,$ one has
$$R_{p,\mu} TK=TR_{p,\mu^T} K.$$
\end{proposition}
\begin{proof}
Suppose $p\in(-1,0)\cup(0,\infty)$; the case $p=0$ follows by continuity. Let $L$ be a star body in $\R^n.$ Then, one can verify from the definition (or see \cite[page 20]{gardner_book}) that
$$\rho_{TL}(x,\theta)=\rho_L(T^{-1}x,T^{-1}\theta).$$
In particular, $\rho_{TL}(\theta)=\rho_L(T^{-1}\theta).$ Then, observe that, by performing the variable substitution $x=Tz,$
\begin{align*}
\rho^{p}_{R_{p,\mu} TK}(\theta)&=\frac{1}{\mu(TK)}\int_{TK}\rho_{TK}(x,\theta)^pd\mu(x)=\frac{1}{\mu(TK)}\int_{TK}\rho_{K}(T^{-1}x,T^{-1}\theta)^pd\mu(x)
\\
&=\frac{1}{\mu^T(K)}\int_{K}\rho_{K}(z,T^{-1}\theta)^pd\mu^T(z)=\rho^{p}_{R_{p,\mu^T} K}(T^{-1}\theta)=\rho^{p}_{TR_{p,\mu^T} K}(\theta).
\end{align*}
\end{proof}

\noindent We will show that when $\mu$ is $s$-concave, $s \geq 0$, then $R_{p,\mu} K$ is a convex body for $p\geq 0$. But first, we must make a detour. 

\subsection{Weighted Projection Bodies}  A convex body $K \in \conbod $ is uniquely determined by its support function $h_K(x)=\sup \{\langle x,y\rangle \colon y \in K \}$. The dual body of $K\in\conbod_0$ is given by $K^\circ=\left\{x\in \R^n: h_K(x) \leq 1\right\}$ (notice this yields that $h_K(x)=\|x\|_{K^\circ}$). For a Borel measure $\mu$ with density $\phi$ and a $K\in\conbod$, the $\mu$-measure of the boundary of $K$, denoted $\partial K$, is 
\begin{equation}\mu^+(\partial K):=\liminf_{\epsilon\to 0}\frac{\mu\left(K+\epsilon B_2^n\right)-\mu(K)}{\epsilon}=\int_{\partial K}\phi(y)d\mathcal{H}^{n-1}(y),
\label{eq_bd}
\end{equation}
where the second equality holds if the $\liminf$ is a limit and there exists some canonical way to select how $\phi$ behaves on $\partial K.$ Here, $\mathcal{H}^{n-1}$ is the $(n-1)$-dimensional Hausdorff measure. It was folklore for quite some time that this formula holds for measures with continuous density (see e.g. K. Ball's work on the Gaussian measure \cite{Ball88}). This was proved rigorously by Livshyts \cite{GAL19}. More recently, it was shown by the first-named author and Kryvonos \cite{KL23} that the formula holds for every Borel measure $\mu$ with density $\phi$, as long as $\phi$ contains $\partial K$ in its Lebesgue set (see Lemma~\ref{l:second} below).

 Using weighted surface area measures, the centered, $\mu$-weighted projection bodies of a convex body $K$ and a Borel measure $\mu$ with continuous density $\phi$ were defined as \cite{LRZ22} the symmetric convex body whose support function is given by, for $\theta\in\s^{n-1}$,
\begin{equation}
    \label{e:mu_polar_suppot}
    h_{\widetilde{\Pi}_\mu K}(\theta)=\frac{1}{2}\int_{\partial K}|\langle\theta,n_K(y) \rangle| \phi(y) d\mathcal{H}^{n-1}(y),
\end{equation}
where $n_K:\partial K \rightarrow \s^{n-1}$ is the Gauss map, which associates an element $y\in\partial K$ with its outer unit normal. Since the set $N_K=\{x\in\partial K: n_K(x) \text{ is not well-defined}\}$ is of measure zero, we will write integrals over $\partial K\setminus N_K$ involving $n_K$ as integrals over $\partial K$ without any confusion. 

As alluded to by the discussion after \eqref{eq_bd}, the formula \eqref{e:mu_polar_suppot} is still well-defined when $\phi$ is not continuous but contains $\partial K$ in its Lebesgue set; in which case, $\phi$ on $\partial K$ is understood as, for $y\in\partial K$, $$\phi(y)=\lim_{\epsilon\to0}\frac{1}{\vol_n(y+\epsilon B_2^n)}\int_{y+\epsilon B_2^n}\phi(x)dx.$$ For such $K$ and $\phi$ ($=$ density of $\mu$), the shift of $K$ with respect to $\mu$ is given by
\begin{equation*}
     \eta_{\mu,K}=\frac{1}{2}\int_{\partial K} n_K(y) \phi(y) d\mathcal{H}^{n-1}(y)=\frac{1}{2}\int_K \nabla \phi(y) d y,
 \end{equation*}
 where the second equality holds when $\phi$ is in $C^1(K).$ Recall the notation that, if $f$ is a function, then there exists two non-negative functions, denoted $f_+$ and $f_-$, such that $f=f_+-f_-$. One can then write $|f|=f_++f_-$ and obtain $|f|-f=2f_-.$ We define the \textit{$\mu$-weighted projection body of $K$} to be the convex body $\Pi_{\mu}K$ defined via the support function, for every $\theta\in\s^{n-1}$
\begin{equation}
\label{eq:mu_weighted_bodies}
\begin{split}
    &h_{\Pi_{\mu}K}(\theta):=h_{\widetilde\Pi_\mu K}(\theta)-\langle\eta_{\mu, K},\theta\rangle
    \\
    &=\frac{1}{2}\int_{\partial K}|\langle\theta,n_K(y) \rangle| \phi(y) d\mathcal{H}^{n-1}(y)-\frac{1}{2}\int_{\partial K} \langle\theta,n_K(y) \rangle \phi(y)d\mathcal{H}^{n-1}(y)
    \\
    &=\int_{\partial K}\langle\theta,n_K(y) \rangle_{-} \phi(y) d\mathcal{H}^{n-1}(y),
    \end{split}
\end{equation}
where the last integral is to emphasize that $\Pi_{\mu}K$ contains the origin in its interior. In the case when $\mu=\lambda_n$, one has
$\Pi K :=\widetilde{\Pi}_{\lambda^n} K=\Pi_{\lambda^n}K,$
where $\Pi K$ is the projection body of $K$. This projection body is a fundamental tool in convex geometry; see e.g. \cite{gardner_book}. It turns out that for $\theta\in\mathbb{S}^{n-1}$:
\begin{equation}
	\label{e:proj_support}
	    h_{\Pi K}(\theta)=\frac{1}{2}\int_{\partial K}|\langle \theta,n_K(y) \rangle| d\mathcal{H}^{n-1}(y)=\vol_{n-1}\left(P_{\theta^{\perp}}K\right),
	\end{equation}
where the first equality is from \eqref{e:mu_polar_suppot}, the orthogonal projection of $K$ onto a linear subspace $H$ is denoted by $P_H K$, and the last equality is known as \textit{Minkowski's projection formula}. We next introduce the weighted covariogram of a convex body.

\subsection{The Covariogram and Radial Mean Bodies} 
\begin{definition}
Let $K$ be a convex body in $\R^n.$ Then, for a Borel measure $\mu$, the $\mu$-\textit{covariogram} of $K$ is the function given by
\begin{equation}
g_{\mu,K}(x)=\mu(K\cap(K+x)).
\label{covario}
\end{equation}
\end{definition}
\noindent The classical covariogram of $K$ is given by $g_K:=g_{\lambda_n,K}$. In \cite{LRZ22}, the following was proven, which extends the volume case first shown by Matheron \cite{MA}. Recall that a domain is an open, connected set with non-empty interior, and that a function $q:\Omega\to\R$ is \textit{Lipschitz} on a bounded domain $\Omega$ if, for every $x,y\in\Omega$, one has
	$|q(x)-q(y)|\leq C|x-y|$ for some $C>0$.
\begin{proposition}[The radial derivative of the covariogram, \cite{LRZ22}]
\label{p:deriv_g_mu_covario}
Let $K\in\conbod$. Suppose $\Omega$ is a domain containing $K$, and consider a Borel measure $\mu$ with density $\phi$ locally Lipschitz on $\Omega$. Then,
\begin{equation}\label{e:deriv_g_mu_covario}
    \diff{g_{\mu,K}(r\theta)}{r}\bigg|_{r=0}=-h_{\Pi_\mu K}(\theta).
    \end{equation}
\end{proposition}
We now briefly show that the assumption of Lipschitz density can be dropped. For a continuous function $h: \s^{n - 1} \to (0, \infty)$, the Wulff shape or Alexandrov body of $h$ is defined as 
$$[h] = \bigcap_{u \in S^{n - 1}} \{x \in \mathbb R^n: \langle x, u\rangle \le h(u)\}.$$
In \cite{KL23}, the first-named author and Kryvonos established the following formula, generalizing the volume case and extending the partial case found in \cite{GAL19}.

\begin{lemma}[Aleksandrov's variational formula for arbitrary measures, \cite{KL23}]
	\label{l:second}
	Let $\mu$ be a Borel measure on $\R^n$ with locally integrable density $\phi$. Let $K$ be a convex body such that $\partial K$, up to set of $(n-1)$-dimensional Hausdorff measure zero, is in the Lebesgue set of $\phi$. Then, for a continuous function $f$ on $\s^{n-1}$, one has that
	$$\lim_{t\rightarrow 0}\frac{\mu([h_K+tf])-\mu(K)}{t}=\int_{\partial K}f(n_K(y))d\mathcal{H}^{n-1}(y).$$
	\end{lemma}	

Next, note that for any $\theta \in \mathbb R^n$, $h_{K + r\theta}(u) = h_K(u) + r\langle u, \theta\rangle$. Also, for any convex body $L$ we have $L = \bigcap_{u \in \s^{n - 1}} \{x: \langle u, x\rangle \le h_L(u)\}$. Consequently,
\begin{align*}
K \cap (K + r \theta) &= \bigcap_{u \in \s^{n - 1}} \{x: \langle u, x\rangle \le h_K(u)\} \cap \bigcap_{u \in \s^{n - 1}} \{x: \langle u, x\rangle \le h_{K + r\theta }(u)\} \\
&= \bigcap_{u \in \s^{n - 1}} (\{x: \langle u, x\rangle \le h_{K}(u)\} \cap \{x: \langle u, x\rangle \le h_{K + r\theta}(u)\}) \\
&= \bigcap_{u \in \s^{n - 1}} \{x:  \langle u, x\rangle \le \min (h_{K}(u), h_K(u) + r\langle \theta, u\rangle)\} \\
&= \bigcap_{u \in \s^{n - 1}} \{x:  \langle u, x\rangle \le h_K(u) + r\min(0, \langle u, \theta\rangle)\}.
\end{align*}
Thus, the body $K_r(\theta) = K \cap (K + r \theta)$ is the Wulff shape of the function $u\mapsto h_K(u) - r \langle u, \theta\rangle_-$. Suppose we have a Borel measure $\mu$ with density $\phi$, such that $\partial K$ is in the Lebesgue set of $\phi$. Then, observe that $g_{\mu,K}(r\theta)=\mu(K_r(\theta))$. Therefore, we obtain \eqref{e:deriv_g_mu_covario} from Lemma~\ref{l:second}, with $f(u) = - \langle u, \theta\rangle_-$, and \eqref{eq:mu_weighted_bodies}. We list this strengthened version as a separate result.
 
\begin{theorem} \label{t:variationalformula}
Let $K$ be a convex body in $\R^n$ and $\mu$ a Borel measure whose density $\phi \colon \R^n \to \R^+$ contains $\partial K$ in its Lebesgue set and $K$ in its support. Then, for every fixed direction $\theta \in \s^{n-1}$, one has 
\begin{equation}\label{e:deriv_g_mu_covario_2}
    \diff{g_{\mu,K}(r\theta)}{r}\bigg|_{r=0}=-h_{\Pi_\mu K}(\theta).
    \end{equation}
\end{theorem}

\noindent From the Brunn-Minkowski inequality, $g_K$ is a $1/n$-concave function supported on $DK$. One can readily check that the $\mu$-covariogram inherits the concavity of the measure $\mu$ in general.
\begin{proposition}[Concavity of the covariogram]
\label{p:covario_concave}
Consider a class of convex bodies $\mathcal{C}\subseteq\conbod$ with the property that $K\in \mathcal{C} \rightarrow K\cap(K+x)\in\mathcal{C}$ for every $x\in DK$. Let $\mu$ be a Borel measure finite on every $K\in\mathcal{C}.$ Suppose $F$ is a continuous and invertible function such that $\mu$ is $F$-concave on $\mathcal{C}$. Then, for $K\in\mathcal{C},$ $g_{\mu,K}$ is also $F$-concave, in the sense that, if $F$ is increasing, then $F\circ g_{\mu,K}$ is concave, and if $F$ is decreasing, then $F\circ g_{\mu,K}$ is convex.
\end{proposition}
\begin{proof}
We first observe the following set inclusion:
for $x, y \in \R^{n}$ and $\lambda \in [0,1]$, we have from convexity that
$$
\begin{aligned}
K \cap(K+(1-\lambda) x+\lambda y) &=K \cap((1-\lambda)(K+x)+\lambda(K+y)) \\
& \supset(1-\lambda)(K \cap(K+x))+\lambda(K \cap(K+y)).
\end{aligned}
$$
Using this set inclusion, we obtain that
$$g_{\mu,K}((1-\lambda) x+\lambda y) \geq \mu((1-\lambda)(K \cap(K+x))+\lambda(K \cap(K+y))).$$
From the fact that $\mu$ is $F$-concave, we obtain
$$
\begin{aligned}
g_{\mu,K}((1-\lambda) x+\lambda y) & \geq F^{-1}\left((1-\lambda) F\left(\mu(K \cap(K+x))\right)+\lambda F\left(\mu(K \cap(K+y))\right)\right) \\
&=F^{-1}\left((1-\lambda) F(g_{\mu,K}(x))+\lambda F(g_{\mu,K}(y))\right).
\end{aligned}
$$
\end{proof}

We see that the ($\mu$-)covariogram connects the ($\mu$-)projection body and difference body of $K$. The covariogram is also related to the radial mean bodies, since, for $p>0$:
\begin{align*}
\int_{K} \rho_{K}(x, \theta)^{p} d \mu(x) &= p\int_K\int_0^{\rho_K(x,\theta)}r^{p-1}dr\,d\mu(x)
\\
&=p \int_{0}^{\rho_{D K}(\theta)}\left(\int_{K \cap(K+r \theta)} d \mu(x)\right) r^{p-1} d r 
\\
&=p \int_{0}^{\rho_{D K}(\theta)} g_{\mu,K}(r \theta) r^{p-1} d r=p\Mel{g_{\mu,K}(r\theta)}{p},
\end{align*}
where, in the second step, we used Fubini's theorem and the fact that $x\in K$ and $-r\theta\in K-x$ for all $0\leq r\leq \rho_K(x,\theta)$. Therefore, we can write, for $p>0$, that
\begin{equation}
\label{eq:equiv}
    \rho_{R_{p,\mu}K}(\theta) \!=\! \left(\frac{p}{\mu(K)}\int_{0}^{\rho_{D K}(\theta)} \!g_{\mu,K}(r \theta) r^{p-1} d r\right)^\frac{1}{p} \!=\! \left(\frac{p}{\mu(K)}\right)^{\frac{1}{p}}\Mel{g_{\mu,K}(r\theta)}{p}^{\frac{1}{p}}.
\end{equation}
Additionally, this formulation implies that $R_{p,\mu}K$ is a convex body if $\mu$ is $s$-concave in the Borell sense for some $s \ge 0$ (see \cite[Theorem 5]{Ball88} for $p\geq 1$ and \cite[Corollary 4.2]{GZ98}). 

\subsection{The Negative Regime for p}
We first convince the reader that $R_{p,\mu} K$ exists for 
 $p\in(-1,0)$ when $\mu$ has a bounded, positive density $\phi$. Under these assumptions, let $M=\min_{x\in K}\phi(x).$ Then, for $p\in (-1,0)\cup(0,\infty):$
\begin{align*}\frac{M}{\|\phi\|_\infty}\frac{1}{\vol_n(K)}\int_K \rho_K(x,\theta)^p dx&\leq \frac{1}{\mu(K)}\int_K \rho_K(x,\theta)^p d\mu(x)
\\
&\leq \frac{\|\phi\|_\infty}{M}\frac{1}{\vol_n(K)}\int_K \rho_K(x,\theta)^p dx. \end{align*}
One then deduces that under these constraints, for $p>0,$ $\left(\frac{M}{\|\phi\|_\infty}\right)^{\frac{1}{p}}R_p K\subseteq R_{p,\mu} K \subseteq \left(\frac{\|\phi\|_\infty}{M}\right)^{\frac{1}{p}}R_p K,$ and, for $p\in (-1,0),$ one has $\left(\frac{M}{\|\phi\|_\infty}\right)^{\frac{1}{p}}R_p K\supseteq R_{p,\mu} K \supseteq \left(\frac{\|\phi\|_\infty}{M}\right)^{\frac{1}{p}}R_p K$. There is equality if, and only if, $\phi$ is constant on $K.$ Notice these inclusions show that $R_{p,\mu} K$ is well-defined for $p\in (-1,0)$. By sending $p\to -1$ we deduce that $R_{p,\mu}K\to \{0\}$ as $p\to -1.$

For a general Borel measure $\mu$ with density, we now obtain a formula for $R_{p,\mu} K$ when $p\in (-1,0).$ This also establishes existence. Notice that, in this instance,
\begin{align*}
\int_{K} \rho_{K}&(x, \theta)^{p} d \mu(x) =-p \int_K \int_{\rho_{K}(x, \theta)}^\infty  r^{p-1}drd\mu(x)
\\
&=-p\int_{0}^{\rho_{DK}(\theta)}\left(\int_{K\setminus {K\cap (K+r\theta)}}d\mu(x)\right)r^{p-1}dr - p\int_K\int_{\rho_{DK}(\theta)}^\infty r^{p-1}drd\mu(x).
\end{align*}
Adding and subtracting integration over $K\cap(K+r\theta),$ we obtain
\begin{align*}
\int_{K} \rho_{K}(x, \theta)^{p} d \mu(x) &=p\int_{0}^{\rho_{DK}(\theta)}(g_{\mu,K}(r\theta)-\mu(K))r^{p-1}dr+\rho^p_{DK}(\theta)\mu(K)
\\
&=p\Mel{g_{\mu,K}(r\theta)}{p}.
\end{align*}
Notice this formulation could have been established directly via the continuity of the Mellin transform. Hence, we can write, for $p\in (-1,0),$ that
\begin{equation}
\label{eq:equiv_neq}
\begin{split}
    \rho_{R_{p,\mu}K}(\theta)&=\left(\frac{p}{\mu(K)}\int_{0}^{\rho_{D K}(\theta)} (g_{\mu,K}(r\theta)-\mu(K)) r^{p-1} d r +\rho^p_{DK}(\theta)\right)^\frac{1}{p}
    \\
    &=\left(\frac{p}{\mu(K)}\right)^{\frac{1}{p}}\Mel{g_{\mu,K}(r\theta)}{p}^{\frac{1}{p}}.
    \end{split}
\end{equation}
The last equality is to emphasis that \eqref{eq:equiv_neq} is the analytic continuation of \eqref{eq:equiv}, as discussed in Section~\ref{sec:ber}.
Now that we have shown the existence of $R_{p,\mu} K$, we can use properties of $p$th averages of functions, i.e. Jensen's inequality, to immediately obtain the following.
\begin{theorem}
\label{t:jensen_set}
Let $\mu$ be a Borel measure finite on a convex body $K$ contained in its support. Then one has that, for $-1 < p \le q \le \infty,$
$$R_{p,\mu} K \subseteq R_{q,\mu} K \subseteq R_{\infty,\mu} K=D K.$$
\end{theorem}

We now take a moment to discuss how $p\in (-1,0)$ was originally handled in the volume case. Gardner and Zhang defined another family of star bodies depending on $K\in\conbod$, the \textit{spectral pth mean bodies} of $K,$ denoted $S_p K.$ However, to apply Jensen's inequality, they had to assume additionally that $\vol_n(K)=1.$ To avoid this assumption, we change the normalization and define $S_p K$ as the star body whose radial function is given by, for $p\in [-1,\infty),$
$$\rho_{S_p K}(\theta)=\left(\int_{P_{\theta^\perp}K} X_\theta K(y)^{p}\left(\frac{X_\theta K(y)dy}{\vol_n(K)}\right)\right)^{1/p},$$
where $X_\theta K(y)=\vol_1(K \cap (y+\theta \R))$ is the \textit{X-ray} of $K$ in the direction $\theta\in\s^{n-1}$ for $y\in P_{\theta^\perp} K$ (see \cite[Chapter 1]{gardner_book} for more on the properties of $X_\theta K$, and note that $\int_{P_{\theta^\perp}K}\frac{X_\theta K(y)dy}{\vol_n(K)}=1$), $\rho_{S_{\infty} K}(\theta) = \max_{y\in \theta^\perp}X_\theta K(y)=\rho_{DK}(\theta),$
$\rho_{S_0 K}(\theta)=\exp\left(\int_{P_{\theta^\perp}K}\log(X_\theta K(y))\frac{X_\theta K(y)dy}{\vol_n(K)}\right),$
and $$\rho_{S_{-1}K}(\theta)=\vol_n(K)\vol_{n-1}(P_{\theta^\perp}K)^{-1}=\vol_n(K)\rho_{\Pi^\circ K}(\theta).$$
Here, $\Pi^\circ K \equiv (\Pi K)^\circ$ is the polar projection body of $K$.

Therefore, from Jensen's inequality, we obtain, for $-1\leq p \leq q \leq \infty,$
\begin{equation}\vol_n(K)\Pi^\circ K=S_{-1}K\subseteq S_p K \subseteq S_q K \subseteq S_{\infty}K=DK.\label{eq:spectral}\end{equation}
The fact that, for $p>-1,$
\begin{equation}\frac{1}{p+1}\int_{P_{\theta^\perp} K}X_\theta K(y)^{p+1}dy=\int_{P_{\theta^\perp} K}\int_0^{X_\theta K(y)}r^pdrdy=\int_K \rho_K(x,\theta)^p dx
\label{eq:radial_spectal_neg}
\end{equation}
shows $S_0 K = e R_0 K$, $S_p K = (p+1)^{1/p}R_p K, \, p>0, $ and that we can analytically continue $R_p K$ to $p\in (-1,0)$ by $R_p K:=(p+1)^{-1/p}S_p K.$ As observed in \cite{GZ98}, the relation $R_p K=(p+1)^{-1/p}S_p K$ shows that $R_{p}K\to \{0\}$ as $p\to -1,$ but the shape of $R_{p} K$ tends to that of $S_{-1} K=\vol_n(K)\Pi^\circ K$ (note that due to the alternate normalization of $S_p K,$ these relations are expressed differently in \cite[Theorem 2.2]{GZ98}; in both instances, it is unknown if $R_p K$ and $S_p K$ are convex for $p\in (-1,0)$).

Gardner and Zhang then obtained a chain of inequalities concerning $R_p K$ (\cite[Theorem 5.5]{GZ98}), a reverse to the one from Jensen's inequality: for $-1<p\leq q < \infty$ 
\begin{equation}
    \label{e:set_inclusion}
    D K \subseteq c_{n, q} R_{q} K \subseteq c_{n, p} R_{p} K \subseteq n \vol_n(K) \Pi^{\circ} K,
\end{equation}
where $c_{n, p}$ are constants defined via continuity at $p=0$, and, for $p\in (-1,0)\cup (0,\infty).$ $
c_{n, p}=(n B(p+1, n))^{-1 / p},
$ with $B(x,y)$ the standard Beta function. There is equality in each inclusion in \eqref{e:set_inclusion} if, and only if, $K$ is a $n$-dimensional simplex. One obtains from \eqref{eq:radial_spectal_neg} that, indeed, $c_{n,p} R_p K$ tends to $n\vol_n(K)\Pi^\circ K$ as $p \to -1^+$, since $c_{n,p}(p+1)^{-\frac{1}{p}}$ tends to $n$. When $p=n$, one obtains, since $\vol_n(R_nK)=\vol_n(K)$, the following special case:
\begin{equation}\label{e:roger_shep_Zhang}
	\frac{\vol_n(DK)}{\vol_n(K)}\leq{2n \choose n} \leq n^n\vol_n(K)^{n-1}\vol_n(\Pi^\circ K).
	\end{equation}
\noindent The left-hand side is the inequality of \textit{Rogers-Shephard inequality} \cite{RS57}, and the right-hand side is \textit{Zhang's inequality} \cite{Zhang91}.

 The goal of the next subsection is to generalize \eqref{e:set_inclusion}. To determine the behaviour of $R_{p,\mu} K$ as $p\to -1^+$, we will not pass through spectral mean bodies. To explain why, we shall, for simplicity, focus on the Gaussian measure and a symmetric $K\in\conbod_0$. Suppose we defined Gaussian spectral mean bodies $S_{p,\gamma_n}K$ as the star body whose radial function is given by, for $p\in [-1,\infty),$
$$\rho_{S_{p,\gamma_n} K}(\theta)=\left(\int_{P_{\theta^\perp}K} \gamma_1(K\cap(y+\theta \R))^{p+1}\left(\frac{d\gamma_{n-1}(y)}{\gamma_n(K)}\right)\right)^{1/p}.$$
Notice that an analogue of \eqref{eq:radial_spectal_neg}, which relates the radial functions of $R_{p}K$ and $S_{p}K$ when $p>-1,$ does not hold. Consequently, we cannot determine the shape of $R_{p,\gamma_n} K$ as $p\to -1$ via $S_{p,\gamma_n} K$. Perhaps then, the focus should be on $S_{p,\gamma_n}K$ and not $R_{p,\gamma_n}K$. But notice that $$\rho_{S_{-1,\gamma_n}(K)}(\theta)=\gamma_n(K)\gamma_{n-1}(P_{\theta^\perp} K)^{-1} \neq \gamma_n(K) \rho_{\Pi^\circ_{\gamma_n} K}(\theta)$$  since one does not have an equivalent of Minkowski's integral formula in the weighted case. Furthermore, it is not necessarily true that $\gamma_{n-1}(P_{\theta^\perp} K)$ is convex as a function of $\theta.$ Hence, it is not necessarily the Minkowski functional of a convex body. Additionally, $\rho_{S_{\infty,\gamma_n} K}(\theta)= \max_{y\in \theta^\perp} \gamma_1(K\cap (y+\theta\R))\neq \rho_{DK}(\theta).$ To summarize, $S_{p,\gamma_n}K$ is not related to $DK$ or $\Pi_{\gamma_n}^\circ K,$ and $R_{p,\gamma_n}K$ is not related to $S_{p,\gamma_n}K.$ It is for these reasons we do not study weighted spectral mean bodies.

We must determine the shape of $R_{p,\mu} K$ as $p\to -1.$ Applying integration by parts to both \eqref{eq:equiv} and \eqref{eq:equiv_neq}, we obtain that, for all $p\in (-1,0)\cup (0,\infty),$ one has
\begin{equation}
    \rho_{R_{p,\mu} K}(\theta)^p=\int_0^{\rho_{DK}(\theta)}\left(\frac{-g_{\mu,K}(r\theta)^\prime}{\mu(K)}\right)r^p dr,
    \label{eq:equiv_new_formula}
\end{equation}
where we used Lebesgue's theorem to obtain that $g_{\mu,K}(r\theta)$ is differentiable almost everywhere on $[0,\rho_{DK}(\theta)],$ as it is monotonically decreasing in the variable $r.$
Taking the limit as $p\to -1,$ we see that $R_{p,\mu} K \to \{o\}.$

On the other-hand, recall the following lemma.
\begin{lemma}[Lemma 4 in \cite{HL22} / Lemma 8 in \cite{HL24}] 
\label{l:fractional_deriv}
If $\varphi:[0, \infty) \rightarrow[0, \infty)$ is a measurable function with $\lim _{t \rightarrow 0^{+}} \varphi(t)=$ $\varphi(0)$ and such that $\int_0^{\infty} t^{-s_0} \varphi(t) \mathrm{d} t<\infty$ for some $s_0 \in(0,1)$, then
$$
\lim _{s \rightarrow 1^{-}}(1-s) \int_0^{\infty} t^{-s} \varphi(t) \mathrm{d} t=\varphi(0) .
$$
\end{lemma}

Therefore, identifying $p=-s$ in Lemma~\ref{l:fractional_deriv}, we obtain from Theorem~\ref{t:variationalformula} that, for $\mu$ with locally integrable density and a convex body $K$ such that $\partial K$ is in the Lebesgue set of the density of $\mu$, one has
\begin{equation}\lim_{p\to -1}(p+1)^{1/p}\rho_{R_{p,\mu} K}(\theta) = \mu(K)\rho_{\Pi^\circ_\mu K}(\theta),
\label{eq:radial_body_limit}
\end{equation}
establishing that the \textit{shape} of $R_{p,\mu} K$ approaches that of $\mu(K) \Pi^\circ_\mu K$ as $p\to -1^+.$

\subsection{Set Inclusions for Weighted Radial Mean Bodies}
\label{sec:radial_bodies}
In this subsection and the next, we obtain the reverse of Theorem~\ref{t:jensen_set} via Berwald's inequality. We will need the following facts about concave functions.

\begin{lemma}
	\label{l:concave}
	Let $f$ be a concave function that is supported on a convex body $L\in\conbod_0$ such that
	$$
h(\theta):=\diff{f(r\theta)}{r}\bigg|_{r=0^+} < 0 \quad  \text{for all }  \theta\in \s^{n-1}.
$$ 
Define $z(\theta)=-\left(h(\theta)\right)^{-1}f(0),$ then
\begin{equation}\label{eq:concave_function} -\infty < f(r\theta)\leq f(0)\left[1-(z(\theta))^{-1}r\right]\end{equation}
whenever $\theta\in\s^{n-1}$ and $r\in[0,\rho_L(\theta)]$. In particular, if $f$ is non-negative, then we have
$$0 \leq f(r\theta)\leq f(0)\left[1-(z(\theta))^{-1}r\right] \quad  \mbox{and }  \rho_L(\theta)\leq z(\theta).
$$
One has $ f(r\theta)=f(0)\left[1-(z(\theta))^{-1}r\right]$ for $r\in[0,\rho_L(\theta)]$ if, and only if, $\rho_L(\theta)=z(\theta).$
	\end{lemma}
 \begin{proof}
	From the concavity of the function $f$, one has
	$$f(r\theta)\leq f(0)\left[1+\frac{h(\theta)}{f(0)}r\right].$$ Then, the inequality \eqref{eq:concave_function} follows from the definition of $z(\theta)$. If $f$ is additionally non-negative, then one obtains that $\rho_L(\theta)\leq z(\theta)$ by using that $0\leq 1-(z(\theta))^{-1}r$ for a fixed $\theta\in\s^{n-1}$ and $r\in[0,\rho_L(\theta)],$ and then setting $r=\rho_L(\theta).$ For the equality conditions, suppose $f(r\theta)=f(0)\left[1-(z(\theta))^{-1}r\right];$ then, by definition of $L$ being the support of $f,$ one obtains $z(\theta)=\rho_L(\theta).$ Conversely, suppose $\rho_L(\theta)=z(\theta).$ Then, from the concavity of $f$, one has $f(0)\left[1-(\rho_L(\theta))^{-1}r\right] \leq f(r\theta) \leq f(0)\left[1-(\rho_L(\theta))^{-1}r\right],$ and so there is equality.
	\end{proof}
		Using Proposition~\ref{p:covario_concave}, Lemma~\ref{l:concave} and \eqref{e:deriv_g_mu_covario_2}, we obtain for a Borel measure $\mu$ with density such that $\mu$ is $F$-concave, $F:\R^+\to \R^+$ is an increasing and differentiable function, that
\begin{equation}
\label{eq:good_set_incl}
DK\subseteq \frac{F(\mu(K))}{F^\prime(\mu(K))}\Pi_\mu^\circ K
\end{equation}
for every $K\in\conbod_0$ such that $\partial K$ is in the Lebesgue set of the density of $\mu$.

\begin{theorem}
\label{t:radial_F_inclusions}
Fix a convex body $K$ in $\R^n$. Let $\mu$ be a finite Borel measure containing $K$ in its support, such that $\mu$ is $F$-concave on convex subsets of $K$, where $F:[0,\mu(K))\to [0,\infty)$ is a continuous, increasing, and invertible function. Then, for $-1<p\le q <\infty$, one has
$$DK\subseteq C(q,\mu,K) R_{q,\mu}K \subseteq C(p,\mu,K) R_{p,\mu} K \subseteq \frac{F(\mu(K))}{F^\prime(\mu(K))}\Pi_\mu^\circ K,$$
where $C(p,\mu,K)=$
$$\begin{cases}
    \left(\frac{p}{\mu(K)}\int_{0}^{1} F^{-1}\left[F(\mu(K))(1-t)\right]t^{p-1}dt\right)^{-\frac{1}{p}} & \text {for } p>0 \\
    \left(\frac{p}{\mu(K)}\int_{0}^{1}t^{p-1} (F^{-1}\left[F(\mu(K))(1-t)\right]-\mu(K))dt+1\right)^{-\frac{1}{p}} & \text {for } p\in (-1,0),
    \end{cases}
    $$
and, for the last set inclusion, we additionally assume that $\mu$ has a locally integrable density containing $\partial K$ in its Lebesgue set and that $F(x)$ is differentiable at the value $x=\mu(K).$
The equality conditions are the following:
\begin{enumerate}
    \item For the first two set inclusions there is equality of sets if, and only if, $F(0)=0$ and $F\circ g_{\mu,K}(x)=F(\mu(K))\ell_{DK}(x).$ 
    \item For the last set inclusion, the sets are equal if, and only if, $F\circ g_{\mu,K}(x)=F(\mu(K))\ell_C(x), \; C=\frac{F(\mu(K))}{F^\prime (\mu(K))}\Pi^\circ_{\mu}K.$
\end{enumerate}

\noindent 
\end{theorem}
\begin{proof}
Observe that $$C(p,\mu,K) \rho_{R_{p,\mu}K}(\theta)=G_{g_{\mu,K}(r\theta)}(p)$$
from \eqref{eq:milman_psi}. Thus, from Lemma~\ref{l:mel_ber}, this function is non-increasing in $p$, which establishes the first three set inclusions. For the last set inclusion, we have not yet established the behaviour of $\lim_{p\to -1}C(p,\mu,K) \rho_{R_{p,\mu}K}(p).$ We do so now.

First, begin by writing
$$G_{g_{\mu,K}(r\theta)}(p)=\frac{C(p,\mu,K)}{(p+1)^{1/p}} (p+1)^{1/p}\rho_{R_{p,\mu}K}(\theta).$$
Therefore, from \eqref{eq:radial_body_limit}, it suffices to show that, as $p\to -1,$ $$\frac{C(p,\mu,K)}{(p+1)^{1/p}} \to \frac{F(\mu(K))}{F^\prime(\mu(K))\mu(K)}.$$
Indeed, from integration by parts we can write, for all $p\in (-1,0)\cup(0,\infty),$ that
$$C(p,\mu,K)=\left(\frac{F(\mu(K))}{\mu(K)}\right)^{-\frac{1}{p}}\left(\int_0^1\left[F^\prime \left(F^{-1}[F(\mu(K))(1-t)]\right)\right]^{-1}t^pdt\right)^{-\frac{1}{p}}.$$
Therefore, the result follows from Lemma~\ref{l:fractional_deriv}.
\end{proof}
\noindent We now obtain a result for $s$-concave measures, $s>0,$ the promised generalization of \eqref{e:set_inclusion}.

\begin{corollary}
\label{cor:set_s_con}
Fix a convex body $K$ in $\R^n$. Let $\mu$ be a Borel measure containing $K$ in its support that is $s$-concave, $s>0,$ on convex subsets of $K$. Then, for $-1< p\leq q < \infty$, one has
$$DK\subseteq {{\frac{1}{s}+q} \choose q}^{\frac{1}{q}} R_{q,\mu}K \subseteq {{\frac{1}{s}+p} \choose p}^{\frac{1}{p}} R_{p,\mu} K\subseteq \frac{1}{s}\mu(K)\Pi_\mu^\circ K,$$
where the last inclusion holds if $\mu$ has locally integrable density $\varphi(x)$ containing $\partial K$ in its Lebesgue set.

\noindent There is equality in any set inclusion if, and only if, $g_{\mu,K}^s(x)=\mu(K)^s\ell_{DK}(x)$.
If $\mu$ is a $s$-concave Radon measure, then $s\in (0,1/n]$ and equality occurs if, and only if, $K$ is $n$-dimensional simplex, the density $\varphi$ of $\mu$ is constant on $K$, and $s = 1/n$.
\end{corollary}
\begin{proof}
Setting $F(x)=x^s$ in Theorem~\ref{t:radial_F_inclusions} yields, in the case when $p>0,$
$$C(p,\mu,K)=\left(p\int_{0}^{1} (1-u)^{1/s}u^{p-1}du\right)^{-\frac{1}{p}}=\left(\frac{p\Gamma(\frac{1}{s}+1)\Gamma(p)}{\Gamma(\frac{1}{s}+p+1)}\right)^{-\frac{1}{p}},$$
and similarly for $p\in (-1,0).$ The equality conditions from Theorem~\ref{t:radial_F_inclusions} yields that $g_{\mu,K}^s(x)$ is an affine function along rays for $x\in DK$. If $\mu$ is a $s$-concave Radon measure, then one must have $s\in (0,1/n]$. For such $s$-concave measures, $g_{\mu,K}^s(x)$ being an affine function along rays is equivalent to the stated equality conditions via Proposition~\ref{p:simp} below.
\end{proof}

We first remark that the following are equivalent:
	\begin{enumerate}
        \item[(i).] $K$ is a simplex.
	    \item[(ii).] \label{item:hom} For any $x \in \R^n$, either $K\cap (K+x)$ is empty or it is homothetic to $K$.
     \end{enumerate}
The equivalence between $(i)$ and $(ii)$ can be found in \cite[Section 6]{EGK64}, or \cite{Choquet,RS57}.  Next, we recall a result of Milman and Rotem \cite[Corollary 2.16]{MR14}:

\begin{lemma}\label{l:milman_rotem} Let $\mu$ be a $s$-concave Radon measure on $\R^n$ with density $\phi$, $A, B \subset \mathbb R^n$ Borel sets of positive measure, and $\lambda \in (0, 1)$, and suppose that 
$$\mu(\lambda A + (1 - \lambda) B)^s = \lambda \mu(A)^s + (1 - \lambda) \mu(B)^s.$$
Then up to $\mu$-null sets, there exist $c, m > 0$, $b \in \R^n$ such that $B = mA + b$ and such that $\phi(mx + b) = c \cdot \phi(x)$ for all $x \in A$. 
\end{lemma}

\begin{proposition}
	\label{p:simp}
	Let $K\in\conbod$, $s\in (0,1/n]$, and $\mu$ an $s$-concave Radon measure, whose density $\varphi$ contains $K$ in its support. Then, $g_{\mu,K}(r\theta)^s$ is an affine function in $r$ for for every $\theta\in\s^{n-1}$ and $r\in[0,\rho_{D K}(\theta)]$ if, and only if, $K$ is $n$-dimensional simplex, $\varphi$ is a constant on $K$, and $s = 1/n$.
	\end{proposition}
 \begin{proof}
    Let $x$ lie in the interior of $\supp (g_{\mu, K}(\cdot))$, and let $t, \lambda \in (0, 1)$. The fact that $g_{\mu, K}(\cdot)$ is affine on the segment $[0, x]$ precisely means that for $\lambda \in (0, 1)$, 
\begin{equation}\label{eq:zhang_eq}
\mu(K^\lambda(0, x))^s = \lambda \mu(K)^s + (1 - \lambda) \mu(K^1(0, x))^s,
\end{equation}
where $K^\lambda(0, y) = K \cap (K + \lambda y)$. Examining the proof of the Proposition~\ref{p:covario_concave}, we see that
$K^\lambda(0, x) \subseteq (1 - \lambda) K + \lambda K^1(0, x)$ and equality can hold in \eqref{eq:zhang_eq} only if $K^\lambda(0, \lambda x) = (1 - \lambda) K + \lambda K^{1}(0, x)$. In particular, we have
\begin{equation}\label{eq:zhang_eq_2}
\mu( (1 - \lambda) K + \lambda K^1(0, x))^s = \lambda \mu(K)^s + (1 - \lambda) \mu(K^1(0, x))^s,
\end{equation}
By Lemma \ref{l:milman_rotem}, that $K \cap (K \cap x)$ is homothetic to $K$ for all $x$ in the interior of $DK$, which implies that $K$ is a $n$-dimensional simplex.

It remains to show that the density of $\mu$ is constant on $K$. For this we use the second conclusion of Lemma \ref{l:milman_rotem}: for each $x \in \mathrm{int}(DK)$ there exists $c(x) > 0$ such that for each $y \in K$, $\varphi(A_x y) = c(x) \varphi(y)$, where $A_x$ is the affine transformation which maps $K$ onto $K \cap (K + x)$. But note that $A_x$ is a continuous map from the compact, convex set $K$ to itself, so it has a fixed point $y$ by Brouwer's fixed point theorem. For such $y$, we have $\varphi(y) = \varphi(A_x y) = c(x) \varphi(y)$, implying $c(x) = 1$. (Note that as a convex function is continuous on the interior of its domain \cite[Theorem 1.5.3]{Sh1}, the density of $\mu$ is continuous on $\mathrm{int}(\supp \mu)$; in particular, $\varphi$ is well-defined pointwise on the interior of $K$, not just up to sets of measure zero.)

Once one knows that $\varphi(A_x y) = \varphi(y)$ for any homothety $A_x$ mapping $K$ to $K \cap (K + x)$ and any $y \in K$, one verifies by tedious but elementary arguments (e.g., by starting with the faces of $K$, working by induction on the dimension) that any two points in $K$ can be mapped into each other by a chain of such $A_x$'s, which implies that $\varphi$ is indeed constant on $K$. Thus we are back to the case of Lebesgue measure, which we know is $1/n$-affine on pairs of homothetic bodies. Conversely, one verifies that if $K$ is a $n$-dimensional simplex and $\phi$ is constant on $K$ then $g_{\mu, K}(\cdot)^{1/n}$ is affine on radial segments, as in the classical Zhang inequality.
 \end{proof}

 Most of the inclusions in Theorem \ref{t:radial_F_inclusions} continue to hold when the concavity of the measures behaves logarithmically. Unfortunately, in this instance, $C(p,\mu,K)$ may tend to $0$ as $p\to \infty,$ and so $C(p,\mu,K)R_{p,\mu}K$ will tend to the origin. Hence, we lose the first set inclusion:
 
\begin{theorem}[Logarithmic Case]
\label{t:set_inclu_log}
Suppose a Borel measure $\mu$ on $\R^n$ is finite on some convex body $K$ and $Q$-concave, where $Q:(0,\mu(K)]\to (-\infty,\infty)$ is an increasing and invertible function. Then, for $-1<p\le q <\infty$, one has
$$C(q,\mu,K)R_{q,\mu}K \subset C(p,\mu,K)R_{p,\mu}K \subset \frac{1}{Q^\prime(\mu(K))}{\Pi_\mu^\circ K},$$
where $C(p,\mu,K)=$
$$\begin{cases}
    \left(\frac{p}{\mu(K)}\int_0^\infty Q^{-1}\left[Q(\mu(K))-t\right]t^{p-1} dt\right)^{-\frac{1}{p}} & \text {for } p>0 \\
    \left(\frac{p}{\mu(K)}\int_{0}^{\infty}t^{p-1} (Q^{-1}\left[Q(\mu(K)-t)\right]-\mu(K))dt\right)^{-\frac{1}{p}} & \text {for } p\in (-1,0),
    \end{cases}$$
and, for the second set inclusion, we additionally assume that $\mu$ has locally integrable density containing $\partial K$ in its Lebesgue set and that $Q(x)$ is differentiable at the value $x=\mu(K)$. In particular, if $\mu$ is log-concave:
$$\frac{1}{\Gamma\left(1+q\right)^{\frac{1}{q}}}R_{q,\mu}K \subset \frac{1}{\Gamma\left(1+p\right)^{\frac{1}{p}}}R_{p,\mu}K \subset \mu(K){\Pi_\mu^\circ K},$$
where $\lim_{p\to 0} \frac{1}{\Gamma\left(1+p\right)^{\frac{1}{p}}}R_{p,\mu}K$ is interpreted via continuity.
\end{theorem}
\begin{proof}
The first inclusion follows from the second case of Theorem~\ref{t:ber}. For the second inclusion, suppose $p>0.$ Then, one has

$$0\le g_{\mu,K}(r\theta)\le Q^{-1}\left[Q(\mu(K))\left(1-\frac{Q^\prime(\mu(K))}{Q(\mu(K))}\frac{r}{\rho_{\Pi_\mu^\circ K}(\theta)}\right)\right].$$
Since $Q(\mu(K))$ may possibly be negative, we shall leave $Q(\mu(K))$ inside the integral:
\begin{align*}
&\rho^p_{R_{p,\mu}K}(\theta)=\frac{p}{\mu(K)}\int_{0}^{\rho_{D K}(\theta)} g_{\mu,K}(r\theta) r^{p-1} d r 
\\
&\leq \frac{p}{\mu(K)}\int_{0}^{\rho_{D K}(\theta)} Q^{-1}\left[Q(\mu(K))\left(1-\frac{Q^\prime(\mu(K))}{Q(\mu(K))}\frac{r}{\rho_{\Pi_\mu^\circ K}(\theta)}\right)\right]r^{p-1} d r.
\\
&=\left(\frac{\rho_{\Pi_\mu^\circ K}(\theta)}{Q^\prime(\mu(K))}\right)^p\frac{p}{\mu(K)}
\\
&\quad\quad\quad\quad\quad\times\int_{0}^{Q^\prime(\mu(K))\frac{\rho_{DK}(\theta)}{\rho_{\Pi_\mu^\circ K(\theta)}}} Q^{-1}\left[Q(\mu(K))-u\right]u^{p-1}du.
\end{align*}
and so
$C(p,\mu,K)\rho_{R_{p,\mu}K}(\theta) < \frac{1}{Q^\prime(\mu(K))}\rho_{\Pi_\mu^\circ K}(\theta),$
which yields the result. The case for $p\in (-1,0)$ is similar.
\end{proof}

We list the Gaussian measure case as a corollary.
\begin{corollary}
Let $K$ be a convex body. Then, for $-1<p\le q <\infty$, one has
$$\frac{1}{\Gamma\left(1+q\right)^{\frac{1}{q}}}R_{q,\gamma_n}K \subset \frac{1}{\Gamma\left(1+p\right)^{\frac{1}{p}}}R_{p,\gamma_n}K \subset \gamma_n(K){\Pi_{\gamma_n}^\circ K},$$ where $\lim_{p\to 0} \frac{1}{\Gamma\left(1+p\right)^{\frac{1}{p}}}R_{p,\gamma_n}K$ is interpreted via continuity, and
$$C(q,\gamma_n,K)R_{q,\gamma_n}K \subset C(p,\gamma_n,K)R_{p,\mu}K \subset \sqrt{\frac{2}{\pi}}e^{-\frac{\Phi^{-1}(\gamma_n(K))^2}{2}}{\Pi_{\gamma_n}^\circ K},$$
where $C(p,\gamma_n,K)=$
$$ \begin{cases}
    \left(\frac{p}{\gamma_n(K)}\int_0^\infty \Phi\left[\Phi^{-1}(\gamma_n(K))-t\right]t^{p-1} dt\right)^{-\frac{1}{p}} & \text {for } p>0 \\
    \left(\frac{p}{\gamma_n(K)}\int_{0}^{\infty}t^{p-1} (\Phi\left[\Phi^{-1}(\gamma_n(K)-t)\right]-\gamma_n(K))dt\right)^{-\frac{1}{p}} & \text {for } p\in (-1,0).
    \end{cases}$$
\end{corollary}

\subsection{Inequalities for Weighted Radial Mean Bodies}
We next show an application of Corollary~\ref{cor:set_s_con}. In particular, if the set inclusions are applied to a measure $\nu$ with homogeneity $\alpha$, then there exists a radial mean body whose $\nu$ measure is ``of the same order" as that of $K$ itself. First, define the \textit{$\nu$-translated-average} of $K$ with respect to $\mu$ as
\begin{equation}
    \nu_\mu(K)=\frac{1}{\mu(K)}\int_K\nu(y-K)d\mu(y)=\frac{1}{\mu(K)}\int_{DK}g_{\mu,K}(x)d\nu(x).
    \label{eq:dual}
\end{equation}
The last equality follows from Fubini's theorem, and this definition has appeared in \cite{LRZ22}. Next, we see that when $\nu$ is homogeneous of degree $\alpha$, we obtain a relation between $\nu(R_{\alpha,\mu} K)$ and $\nu_{\mu}(K).$
\begin{lemma}
\label{l:radial_meas}
Fix a convex body $K$ and a Borel measure $\nu$ that is $\alpha$-homogeneous with density and a Borel measure $\mu$ on $\R^n$. Then, one has $\nu(R_{\alpha,\mu} K)=\nu_{\mu}(K)$.
\end{lemma}
\begin{proof}
Let $\varphi$ be the density of $\nu$. Using \eqref{eq:star_form} and Fubini's theorem, we obtain:
\begin{align*}
\nu(R_{\alpha,\mu}K)&=\frac{1}{\alpha}\int_{\s^{n-1}}\rho^\alpha_{R_{\alpha,\mu}K}(\theta)\varphi(\theta)d\theta=\frac{1}{\alpha}\frac{1}{\mu(K)}\int_{\s^{n-1}}\int_K\rho_K(x,\theta)^\alpha d\mu(x)\varphi(\theta)d\theta
\\
&=\frac{1}{\alpha}\frac{1}{\mu(K)}\int_K\int_{\s^{n-1}}\rho_K(x,\theta)^\alpha \varphi(\theta)d\theta d\mu(x)
\\
&=\frac{1}{\alpha}\frac{1}{\mu(K)}\int_K\int_{\s^{n-1}}\rho_{x-K}(\theta)^\alpha \varphi(\theta)d\theta d\mu(x),
\end{align*}
where the last equality follows from the fact that $\rho_K(x,\theta)=\rho_{K-x}(-\theta)=\rho_{x-K}(\theta).$ Using \eqref{eq:star_form} again yields the result.
\end{proof}

\begin{theorem}[Rogers-Shephard type inequality for an $\alpha$-homogeneous and a $s$-concave measure]
\label{t:rsas}
Fix a convex body $K$. Consider a Borel measure $\nu$ that is $\alpha$-homogeneous and a Borel measure $\mu$ on $\R^n$ that is $s$-concave, $s>0$, on convex subsets of $K$, whose locally integrable density contains $\partial K$ in its Lebesgue set and $K$ in its support. Then,
$$\nu(DK)\leq {{\frac{1}{s}+\alpha} \choose \alpha} \min\{\nu_{\mu}(K),\nu_{\mu}(-K)\},$$
with equality if, and only if, $g_{\mu,K}^s(x)=\mu(K)^s\ell_{DK}(x)$. If $\mu$ is a $s$-concave Radon measure, then $s\in (0,1/n]$ and equality occurs if, and only if, $K$ is $n$-dimensional simplex, the density $\varphi$ of $\mu$ is constant on $K$, and $s = 1/n$.
\end{theorem}
\begin{proof}
From Corollary~\ref{cor:set_s_con} with $p=\alpha$ one obtains
$$\nu(DK)\le \nu\left({{\frac{1}{s}+\alpha} \choose \alpha}^{\frac{1}{\alpha}} R_{\mu,\alpha}K\right)={{\frac{1}{s}+\alpha} \choose \alpha} \nu(R_{\mu,\alpha}K).$$
Using Lemma~\ref{l:radial_meas} and that $DK=D(-K)$ completes the proof.
\end{proof}
An upper bound for $\mu(DK)/\mu(K)$ when $\mu$ is $s$-concave was first shown by Borell, \cite{Bor75}. However, the bound was not sharp.

\begin{corollary}[Zhang's Inequality for an $\alpha$-homogeneous and a $s$-concave measure]
\label{cor:zhang_s}
Fix a convex body $K$. Consider a Borel measure $\nu$ that is $\alpha$-homogeneous and a Borel measure $\mu$ on $\R^n$ that is $s$-concave, $s>0$, on convex subsets of $K$, whose locally integrable density contains $\partial K$ in its Lebesgue set and $K$ in its support. Then, one has
$$s^\alpha{{\frac{1}{s}+\alpha} \choose \alpha} \le \frac{\mu(K)^{\alpha}}{\nu_\mu(K)}\nu\left(\Pi_{\mu}^\circ K\right),$$
with equality if, and only if, $g_{\mu,K}^s(x)=\mu(K)^s\ell_{\Pi_{\mu}^\circ K}(x)$. If $\mu$ is a $s$-concave Radon measure, then $s\in (0,1/n]$ and equality occurs if, and only if, $K$ is a $n$-dimensional simplex, the density $\varphi$ of $\mu$ is constant on $K$, and $s = 1/n$.
\end{corollary}
\begin{proof}
From Lemma~\ref{l:radial_meas} and Corollary~\ref{cor:set_s_con} with $p=\alpha,$ one obtains
\begin{align*}{{\frac{1}{s}+\alpha} \choose \alpha} \nu_{\mu}(K) &= {{\frac{1}{s}+\alpha} \choose \alpha} \nu(R_{\mu,\alpha}K)=\nu\left({{\frac{1}{s}+\alpha} \choose \alpha}^{\frac{1}{\alpha}} R_{\mu,\alpha}K\right)
\\
&\le\nu\left(\frac{1}{s}\mu(K)\Pi_\mu^\circ K\right).\end{align*}
\end{proof}

\subsection{The Gardner-Zvavitch Inequality and Radial Mean Bodies}
We would like to apply the $1/n$-concavity of the Gaussian measure over symmetric convex bodies in \eqref{e:gamma_gaussian}, which is also true for every $\mu\in\mathcal{M}_n$, to obtain that Corollary~\ref{cor:set_s_con} holds for such measures. We run into an issue: even if $K$ is symmetric, then $K\cap (K+x)$ is \textit{not} symmetric in general. Therefore, Proposition~\ref{p:covario_concave} does not apply, i.e. $\mu$ being $1/n$ concave does not imply that $g_{\mu,K}^\frac{1}{n}$ is concave. To remedy this, we take a cue from \cite{LRZ22} and define the polarized covariogram as
$$r_{\mu,K}(x)=\mu\left((K+\frac{x}{2})\cap(K-\frac{x}{2})\right).$$
As can be found in \cite{LRZ22}, the set $(K+\frac{x}{2})\cap(K-\frac{x}{2})$ is a symmetric convex body when $K$ is, and $r_{\mu,K}$ inherits any concavity of $\mu$ over symmetric convex bodies. 

Under the assumption that $\mu$ is a Borel measure with even density and $K\in\conbod$ is symmetric, notice that $\Pi_\mu K = \widetilde{\Pi}_\mu K$. One also has, for every $\theta\in\s^{n-1}$,
\begin{equation}\label{e:deriv_g_mu_covario_polar}
    \diff{r_{\mu,K}(r\theta)}{r}\bigg|_{r=0^+}=-h_{\Pi_\mu K}(\theta).
    \end{equation}
This was first shown in \cite{LRZ22} under the additional assumption that the density of $\mu$ is Lipschitz; arguing similarly to Theorem~\ref{t:variationalformula} allows one to weaken the assumption to merely the density of $\mu$ contains $\partial K$ in its Lebesgue set. In fact, one does not need the symmetry assumptions on $\mu$ and $K$ for this proof; one will obtain in general
\begin{equation}
\label{e:deriv_g_mu_covario_polar_gen}
    \diff{r_{\mu,K}(r\theta)}{r}\bigg|_{r=0^+}=-h_{\widetilde\Pi_\mu K}(\theta),
    \end{equation}
but $\widetilde\Pi_\mu K$ is not necessarily $\Pi_\mu K$.

In order to obtain Corollary~\ref{cor:set_s_con} for $\mu\in\mathcal{M}_n$, we define the \textit{polarized weighted mean bodies}, $P_{p,\mu} K$, as the star bodies on $\R^n$ whose radial function is given by, for $p\in (-1,\infty)$ and $\theta\in\s^{n-1}$,
\begin{equation}
\label{eq:polarized radial}
    \rho_{P_{p,\mu}K}(\theta)^p=\left(\frac{p}{\mu(K)}\right)\Mel{r_{\mu,K}(r\theta)}{p}=\int_0^{\rho_{DK}(\theta)}\left(\frac{-r_{\mu,K}(r\theta)^\prime}{\mu(K)}\right)r^p dr.
\end{equation}
The bodies $P_{p,\mu} K$ are symmetric convex bodies for $p \geq 0$. Once again, $p=0,\infty$ are interpreted via continuity, and $R_{\infty,\mu} K=DK$ for $K$ contained in the support of $\mu$. We again have that
\begin{equation}\lim_{p\to -1}(p+1)^{1/p}\rho_{P_{p,\mu} K}(\theta) = \mu(K)\rho_{\widetilde\Pi^\circ_\mu K}(\theta).
\label{eq:radial_body_limit_polar}
\end{equation}
Consequently, the proof of the following theorem is verbatim the same as Theorem~\ref{t:radial_F_inclusions} and Corollary~\ref{cor:set_s_con}.
\begin{theorem}
\label{t:set_s_con}
Fix a symmetric convex body $K$ in $\R^n$. Let $\mu\in\mathcal{M}_n$ be a Borel measure containing $K$ in its support. Then, for $-1< p\leq q < \infty$, one has
$$DK\subseteq {{n+q} \choose q}^{\frac{1}{q}} P_{q,\mu}K \subseteq {{n+p} \choose p}^{\frac{1}{p}} P_{p,\mu} K\subseteq n\mu(K)\Pi_\mu^\circ K.$$

\noindent There is equality in any set inclusion if, and only if, $r_{\mu,K}(x)=\mu(K)\ell_{DK}(x)^n$.
\end{theorem}

Recently, it was shown by Livshyts \cite{GL23} that for any even, log-concave probability measure $\mu$ on $\R^n$, $\mu$ is $s$-concave over the class of symmetric convex bodies, with $s = n^{-4-s_n}$. One can then use this result to formulate Theorem~\ref{t:set_s_con} for such measures, with $n$ in the coefficients replaced by $n^{4+s_n}$. Here, $\{s_n\}$ is a bounded sequence that goes to $0$ as $n\to \infty$.

It is also manifest that Theorem~\ref{t:radial_F_inclusions} holds with the weighted radial mean bodies replaced by the polarized weighted mean bodies, the additional assumption that $K$ is symmetric (since $\mu$ being $F$-concave on convex subsets of $K$ implies it is $F$-concave on symmetric convex subsets of $K$, and the fact that $r_{\mu,K}$ will then inherit the concavity) and $\Pi^\circ_\mu K$ replaced by $\widetilde\Pi^\circ_\mu K$ (due to \eqref{e:deriv_g_mu_covario_polar_gen} and \eqref{eq:radial_body_limit_polar}). We avoid the unnecessary formal statement of this slightly different theorem.

\printbibliography

\end{document}